\documentclass[twocolumn]{autart}    

\usepackage{savesym} 
\savesymbol{AND} 
\savesymbol{OR} 
\savesymbol{NOT} 
\savesymbol{TO} 
\savesymbol{COMMENT} 
\savesymbol{BODY} 
\savesymbol{IF} 
\savesymbol{ELSE} 
\savesymbol{ELSIF} 
\savesymbol{FOR} 
\savesymbol{WHILE}
\usepackage{amsmath,amssymb,graphicx,tikz,enumerate,xcolor,colortbl,multirow,algorithm,algorithmic,float,lipsum}         
\usepackage[colorlinks,linkcolor=blue, anchorcolor=blue,citecolor=green,urlcolor=brown,hyperindex]{hyperref}

\usepackage{stmaryrd}

\makeatletter
\newcommand*{\rom}[1]{\expandafter\@slowromancap\romannumeral #1@}
\makeatother

\makeatletter
\newenvironment{breakablealgorithm}
{
	\begin{center}
		\refstepcounter{algorithm}
		\hrule height.8pt depth0pt \kern2pt
		\renewcommand{\caption}[2][\relax]{
			{\raggedright\textbf{\ALG@name~\thealgorithm} ##2\par}%
			\ifx\relax##1\relax 
			\addcontentsline{loa}{algorithm}{\protect\numberline{\thealgorithm}##2}%
			\else 
			\addcontentsline{loa}{algorithm}{\protect\numberline{\thealgorithm}##1}%
			\fi
			\kern2pt\hrule\kern2pt
		}
	}{
		\kern2pt\hrule\relax
	\end{center}
}
\makeatother

\newcommand{\Z}{\mathbb{Z}}
\newcommand{\R}{\mathbb{R}}
\newcommand{\N}{\mathbb{N}}

\newcommand{\LM}{\mathcal{L}}

\newcommand{\dt}{\delta}
\newcommand{\Dt}{\Delta}

\newcommand{\D}{\mathcal{D}}

\newcommand{\E}{\mathcal{E}}

\newcommand{\V}{\mathcal{V}}
\newcommand{\Gr}{\mathcal{G}}
\newcommand{\W}{\mathcal{W}}

\newcommand{\Ccal}{\mathcal{C}}
\newcommand{\Sig}{\Sigma}

\newcommand{\col}{\operatorname{Col}}

\newcommand{\Prob}{\operatorname{Prob}}
\newcommand{\outdeg}{\operatorname{outdeg}}

\newcommand{\red}{\color{red}}
\newcommand{\blue}{\color{blue}}
\definecolor{green}{rgb}{0.1,0.7,0.1}
\newcommand{\green}{\color{green}}
\newcommand{\brown}{\color{brown}}

\newcommand{\algorithmicinit}{\textbf{initialization: }}
\newcommand{\init}{\STATE \algorithmicinit}
\newcommand{\algorithmicstop}{\textbf{stop}}
\newcommand{\STOP}{\STATE \algorithmicstop}

\newcommand{\PTIME}{\mathsf{P}}

\newcommand{\NP}{\mathsf{NP}}

\newcommand{\llb}{\llbracket}
\newcommand{\rrb}{\rrbracket}

\newcounter{enumi_saved}
\newenvironment{myenumerate} {
    \begin{enumerate}[(i)]\setcounter{enumi}{\value{enumi_saved}}}
    {\setcounter{enumi_saved}{\value{enumi}}\end{enumerate}}

\newtheorem{theorem}{Theorem}[section]
\newtheorem{lemma}[theorem]{Lemma}

\newtheorem{definition}[theorem]{Definition}
\newtheorem{proposition}[theorem]{Proposition}

\newtheorem{example}[theorem]{Example}

\newtheorem{remark}{Remark}[section]

\allowdisplaybreaks

\usetikzlibrary{shapes,arrows,positioning}

\begin{document}

\begin{frontmatter}

\title{Synthesis for observability of logical control networks\thanksref{footnoteinfo}} 

\thanks[footnoteinfo]{
The author discussed these issues with Dr. Karl Henrik Johansson at KTH Royal Institute of Technology,
and is grateful to him for his suggestions.
}

\author[SurreyU]{Kuize Zhang}\ead{kuize.zhang@surrey.ac.uk}    

\address[SurreyU]{
Department of Computer Science, University of Surrey, Guildford GU2 7XH, UK}

\begin{keyword}                           
	logical control network,  observability, state feedback, synthesis, semitensor product, observability graph

\end{keyword}                             

\begin{abstract}                          
	Finite-state systems have applications in systems biology, formal verification and synthesis 
	of infinite-state (hybrid) systems, etc.
	As deterministic finite-state systems, logical control networks (LCNs) consist of a finite number of nodes
	which can be in a finite number of states and update their states.
	In this paper, we investigate the synthesis problem for observability 
	of LCNs based on state feedback with exogenous input by using the semitensor product proposed by 
	Daizhan Cheng and the notion of observability graph (previously called weighted pair graph) proposed
	by us.
	We show that state feedback with exogenous input can either enforce or weaken observability of an LCN.
	We prove that for an LCN $\Sig$ and another closed-loop 
	LCN $\Sig_{\Ccal}$ obtained by feeding
	a state-feedback controller $\Ccal$ with exogenous input into $\Sig$,
	(1) if $\Sig$ is observable, then $\Sig_{\Ccal}$ can be either observable or not;
	(2) if $\Sig$ is not observable, $\Sig_{\Ccal}$ can also be observable or not. 
	We also prove that if an unobservable LCN can be made observable by state feedback
	with exogenous input, then it can also be made observable by state feedback 
	(without exogenous input, equivalent to state feedback with constant input).
	Furthermore, we give an upper bound on the number of state-feedback controllers that are needed
	to be tested in order to verify whether an unobservable LCN can be 
	made observable by state feedback, and based on the procedure of obtaining the upper bound,
	we design an observability synthesis algorithm, by additionally combining the ideas of a greedy
	algorithm and dynamic programming. These results open the study of observability synthesis in LCNs.
\end{abstract}

\end{frontmatter}

\section{Introduction}

\subsection{Background}

Finite-state systems have applications in many areas such as formal verification and synthesis 
of infinite-state (hybrid) systems \cite{tab09,Belta2017FormalMethodsDTDS,Baier2008ModelChecking},
systems biology \cite{Akutsu2018BNbook}, etc.

As special deterministic finite-state systems in which all nodes can be only in one of two states,
Boolean control networks (BCNs) were proposed to describe genetic regulatory networks
\cite{Kauffman1969RandomBN,Ideker2001}. 
In a BCN, nodes
can be in one of two discrete states ``$1$'' and ``$0$'', which represent
a gene state ``on'' (high concentration of a protein) and 
``off'' (low concentration), respectively.
Every node updates its state according to a Boolean function of the states of several of the network nodes. 
Although BCNs are a simplified model of genetic regulatory networks, 
they can be used to characterize many
important phenomena of biological systems, e.g., cell cycles \cite{Faure2006DynamicalAnalysisofGBMforMCC},
cell apoptosis \cite{Sridharan2012}. Hence the study on BCNs has
been paid wide attention \cite{Kitano2002,Akutsu2000,Albert2000,Zhao2013AggregationAlgorithmBNattractor}.

A logical control network (LCN) is also a deterministic finite-state system but naturally extends a BCN
in the sense that the nodes of the former can be in one of a finite number (but not necessarily $2$) 
of states \cite{Zhao2011OptimalControlLCN}. From a practical point of view, LCNs can be used to describe 
more biological systems than BCNs. However, under the semitensor product (STP) framework, they have the same
algebraic form \cite{Zhao2011OptimalControlLCN}, hence they can be dealt with by using the same method.
In this paper, we study LCNs.

In 2007, Akutsu et al. \cite{Akutsu2007} proved that it is $\NP$-hard to verify whether a BCN is
controllable in the number of nodes, hence
there exists no polynomial-time algorithm for determining controllability of BCNs
unless  $\PTIME=\NP$. They also
pointed out  that ``One of the major goals of
systems biology is to develop a control theory for complex biological systems''.
Later in 2013, Laschov et al. \cite{Laschov2013ObservabilityofBN:GraphApproach} proved that
it is also $\NP$-hard to verify whether a BCN is observable in the number of nodes.
These $\NP$-hardness results show that it is generally intractable to verify controllability
and observability for
large-scale BCNs (those with more than approximately $30$ nodes, i.e., more than $2^{30}$ states),
and also stimulate explorations on how to design efficient verification algorithms that work on large-scale 
BCNs with special structures, e.g., in \cite{Zhao2016ControllabilityAggregationBCN,Zhang2020ObsLargeBCN}.

Recently, a control-theoretic framework for BCNs based on the STP of matrices
(proposed by Cheng \cite{Cheng2001STP} in 2001) was established by
Cheng and Qi \cite{Cheng2009bn_ControlObserva} in 2009. Although the STP method
cannot make the generally intractable problems related to BCNs become tractable, it provides 
a matrix method for characterizing BCNs, so that many matrix-based techniques in control theory 
can be used to study BCNs. As a result, it stimulates the studies on control problems of
BCNs based on diverse methods, e.g., controllability
\cite{Cheng2009bn_ControlObserva,Zhao2010InputStateIncidenceMatrix},
observability \cite{Cheng2009bn_ControlObserva,Fornasini2013ObservabilityReconstructibilityofBCN,Zhang2016ObservabilityofBCN,Li2015ControlObservaBCN},
reconstructibility \cite{Fornasini2013ObservabilityReconstructibilityofBCN,Zhang2016WPGRepresentationReconBCN},
identifiability \cite{Cheng2011IdentificationBCN,Zhang2017IdentificationBCN},
invertibility \cite{Zhang2015Invertibility_BCN},
Kalman decomposition \cite{Zou2015KalmanDecompositionBCN}, disturbance decoupling 
\cite{Li2021EquivalentConditionDDP_BCN}, and other related results
\cite{Li2016PinningControlSynchronizationBCN,Wu2018FiniteConvergenceStoLDS,Guo2018SetStabilityPBN,Li2018BisimulationBCN,Li2013OututFeedbackBCN}. 
Among these results, some was worked out based on a computational-algebra method
\cite{Li2015ControlObservaBCN},
some based on finite automata and graph theory \cite{Zhang2016ObservabilityofBCN,Zhang2016WPGRepresentationReconBCN},
some based on the STP and graph theory \cite{Fornasini2013ObservabilityReconstructibilityofBCN},
and some based on symbolic dynamics \cite{Zhang2015Invertibility_BCN,Hochma2013SymbolicdynamicsofBCN}.

\subsection{Literature review}
\label{subsec:LiterRev}

Among many control properties, \emph{controllability} and \emph{observability} are the most
fundamental ones. The former implies that an arbitrary given state of a system can be steered to an 
arbitrary given state by some input sequence. 
The latter implies that the initial state can be determined by a sufficiently long input sequence 
and the corresponding output sequence. 
The importance of controllability and observability of BCNs can be found in \cite{Akutsu2007} and
\cite{Sridharan2012}, etc. 
Lack of these properties makes a system lose many good behaviors. 
So, it is important to investigate 
how to enforce controllability and observability, e.g., by means of feedback and controller synthesis.
Matrix forms of necessary and sufficient conditions for controllability of BCNs were given in 
\cite{Cheng2009bn_ControlObserva,Zhao2010InputStateIncidenceMatrix}. Necessary and sufficient  conditions
for observability of BCNs are much more difficult to obtain, and furthermore, there exist nonequivalent
definitions of observability in BCNs, e.g., it was proven in \cite{Zhang2016ObservabilityofBCN} that 
four definitions of observability are pairwise nonequivalent, showing that observability is not dual 
to controllability, remarkably differentiating BCNs from linear control systems. A summary on 
necessary and sufficient conditions of different definitions of observability of BCNs is given in 
Table~\ref{tab1:ObsResulBCN}.
In \cite{Zhang2014ObservabilityofBCNCCC,Zhang2016ObservabilityofBCN}, a notion of \emph{weighted
pair graph} was proposed (later it was renamed \emph{observability graph} in
\cite{Zhang2020ObsLargeBCN,Zhang2020bookDDS}), and verifiable necessary and sufficient conditions for four definitions
of observability (shown in Definitions~\ref{def1:observability_aggregation}, \ref{def4_observability},
\ref{def1_observability}, \ref{def7_observability})
were obtained by computing different types of deterministic finite automata from an observability graph
adapted to the four definitions of observability. Later on, the observability graph was used in many papers
to solve related problems, e.g.,
\cite{Cheng2016NoteonObservabilityBCN} (the observability graph was called \emph{observability matrix} therein,
the set of diagonal vertices
and the set of non-diagonal vertices in an observability graph (see Definition~\ref{def2:observability_aggregation})
are exactly the set $D$ and the set $\Xi$ in \cite[Page 78]{Cheng2016NoteonObservabilityBCN}),
\cite{Zhu2018ObservabilityBCN,Zhang2020ObsLargeBCN} (called \emph{observability graph} therein),
\cite{Cheng2018ObservabilityBCNSetControl,Wang2021PerburbationAnaObservabilityBCN}\footnote{
\cite[Theorem~1]{Wang2021PerburbationAnaObservabilityBCN} is a special case of 
\cite[Theorem~3.15]{Zhang2016ObservabilityofBCN} (i.e., Lemma~\ref{lem1:observability_aggregation}).},
\cite{Guo2018ObservabilityBCN,Zhou2019ObservabilityPBN} (called \emph{parallel extension} therein),
\cite{Yu2021ObservabilityBooleanNetwork}. In addition, the non-diagonal subgraph of an observability graph
(see Definition~\ref{def2:observability_aggregation}) was proposed in \cite{Zhang2016WPGRepresentationReconBCN}
(called weighted pair graph therein, later renamed \emph{reconstructibility graph} in 
\cite{Zhang2020ObsLargeBCN} and \emph{detectability graph} in \cite{Zhang2020bookDDS}) to 
verify two definitions of reconstructibility of BCNs (the stronger one was earlier studied in
\cite{Fornasini2013ObservabilityReconstructibilityofBCN}).
Recently, a variant of the reconstructibility graph was used to 
study reconstructibility of singular Boolean control networks \cite{Li2020ReconstructibilitySBCN},
where such networks are a subclass of nondeterministic finite-transition systems \cite{Zhang2018ObservabilityFLTS}
by definition.

Observability results have also been extended to nondeterministic finite-transition systems (NFTS's)
\cite{Zhang2018ObservabilityFLTS} and probabilistic Boolean networks (PBNs)
\cite{Zhao2015ObservabilityPBN,Zhou2019ObservabilityPBN,Fornasini2020ObserReconPBN,Yu2021ObservabilityBooleanNetwork},
where the stochastic switching signals in PBNs are independent and identically distributed 
processes, hence the PBNs are actually discrete-time finite-state time-homogeneous 
Markov chains. Moreover, if all probabilities in such a PBN are removed then it
becomes an NFTS in which the input is constant. That is, the systems considered in 
\cite{Zhang2018ObservabilityFLTS} are more general than the systems considered in 
\cite{Zhao2015ObservabilityPBN,Zhou2019ObservabilityPBN,Fornasini2020ObserReconPBN,Yu2021ObservabilityBooleanNetwork}.
In \cite{Zhang2018ObservabilityFLTS}, the notion of observability graph was extended from BCNs
to NFTS's, where the matrix $\mathsf{O}$ in Eqn.~(5) of 
\cite{Zhang2018ObservabilityFLTS} is the adjacency matrix of the (extended) observability graph 
of an NFTS.
By additionally computing different types of deterministic finite automata from an observability graph,
verifiable necessary and sufficient conditions were given for
Definitions~\ref{def1:observability_aggregation}, \ref{def4_observability}, \ref{def7_observability} extended
to NFTS's. In \cite{Zhao2015ObservabilityPBN,Zhou2019ObservabilityPBN,Fornasini2020ObserReconPBN,Yu2021ObservabilityBooleanNetwork},
three definitions of observability for PBNs were studied: observability in probability on $\llb 0,\theta \rrb$ with
$\theta\in\N$ (see notation in Section~\ref{subsec:STP}), finite-time observability in probability,
and asymptotic observability in distribution. The main results in
\cite{Zhou2019ObservabilityPBN,Yu2021ObservabilityBooleanNetwork} were obtained
by using the observability graph (called \emph{parallel extension} in \cite{Zhou2019ObservabilityPBN}).
Later, by definition we will show that the first two of the above three definitions of observability in PBNs
are actually a slightly stronger version of Definition~\ref{def4_observability} and
Definition~\ref{def4_observability} itself, in BCNs, respectively. We will also point out that the
results in \cite{Yu2021ObservabilityBooleanNetwork} already show that 
the third one is also a slightly stronger version of Definition~\ref{def4_observability} of BCNs.
That is, the probabilities in the PBNs studied in
\cite{Zhao2015ObservabilityPBN,Zhou2019ObservabilityPBN,Fornasini2020ObserReconPBN,Yu2021ObservabilityBooleanNetwork}
play no role in adding stochasticity when studying observability. See Remark~\ref{rem6:synthesisConObser},
in which the \emph{deterministic essence} of observability of PBNs studied in 
\cite{Zhao2015ObservabilityPBN,Zhou2019ObservabilityPBN,Fornasini2020ObserReconPBN,Yu2021ObservabilityBooleanNetwork}
is revealed.
	\begin{table*}[!htbp]
				\centering
		{\scriptsize
		\begin{tabular}{|c||c|c|c|c|}
			\hline\rowcolor{lightgray}
			& Def.~\ref{def1_observability} & Def.~\ref{def4_observability} & Def.~\ref{def7_observability} & Def.~\ref{def1:observability_aggregation}\\\hline
			Fornasini and Valcher \cite{Fornasini2013ObservabilityReconstructibilityofBCN} & & & & {\brown $O(2^{4n+m})$} \\\hline
			Li et al. \cite{Li2014ObservabilityConditionsofBCN} & & & {\green $O(2^{2^{2n}+m})$} &\\\hline

			Zhang and Zhang \cite{Zhang2014ObservabilityofBCNCCC} & \multirow{4}{*}{{\blue $O(2^{n+2^{2n}+m})$}} & 
			\multirow{4}{*}{{\blue $O(2^{4n+m})$}} & \multirow{4}{*}{{\blue$O(2^{2^{2n}+m})$}}
			& \multirow{4}{*}{{\blue $O(2^{2n+m})$}}\\
			Zhang and Zhang \cite{Zhang2016ObservabilityofBCN} & & & & \\
			{\blue weighted pair graph} 
			{\blue (WPG, $O(2^{2n+m})$)} & & & &\\
			{\blue (renamed observability graph in \cite{Zhang2020ObsLargeBCN,Zhang2020bookDDS})} & & & & \\\hline

			Li et al. \cite{Li2015ControlObservaBCN} & \multirow{3}{*}{} & \multirow{3}{*}{\red $O(2^{2^{2n}+m})$} & \multirow{3}{*}{} & \multirow{3}{*}{} \\
			{\red computational algebra} & & & & \\
			{\red (very fast in sparse BCNs)} & & & & \\\hline

			Cheng et al. \cite{Cheng2016NoteonObservabilityBCN} & & \multirow{3}{*}{\blue $O(2^{2n+m})$} & & \\
			{\blue observability matrix} & & & &\\
			{\blue (i.e., adjacency matrix of WPG)} & & & &\\\hline

			Zhu et al. \cite{Zhu2018ObservabilityBCN} & & \multirow{2}{*}{\blue $O(2^{2n+m})$} & & \\
			{\blue observability graph (i.e., WPG)} & & & & \\\hline

			Cheng et al. \cite{Cheng2018ObservabilityBCNSetControl} & & \multirow{2}{*}{\blue $O(2^{6n+m})$} & & \\
			{\blue set controllability}  & & & & \\\hline

			Guo \cite{Guo2018ObservabilityBCN} & & \multirow{3}{*}{\blue $O(2^{6n+m})$} & \multirow{3}{*}{\blue $O(2^{n2^{2n+1}+m})$} & \multirow{3}{*}{\blue $O(2^{n2^{2n+1}+m})$} \\
			{\blue parallel extension} & & & & \\
			{\blue (i.e., adjacency matrix of WPG)}  & & & & \\\hline
		\end{tabular}
		}
		\caption{Complexity upper bounds for verifying four definitions of observability in BCNs,
			the same color represents equivalent methods, $n$ and $m$ denote the numbers of state nodes and
			input nodes. The observability matrix in \cite{Cheng2016NoteonObservabilityBCN}
			is exactly the adjacency matrix of the weighted pair graph proposed in 
			\cite{Zhang2014ObservabilityofBCNCCC}, the observability graph in \cite{Zhu2018ObservabilityBCN}
			is exactly the weighted pair graph, the parallel extension in \cite{Guo2018ObservabilityBCN}
			is actually the adjacency matrix of the weighted pair graph when being used to verify 
			observability, the parallel extension was later used to verify observability
			of probabilistic Boolean networks in \cite{Zhou2019ObservabilityPBN},
			the set-controllability method in \cite{Cheng2018ObservabilityBCNSetControl} is equivalent to
			the parallel-extension method in \cite{Guo2018ObservabilityBCN} when being applied to 
			verify observability. The method in 
			\cite{Li2015ControlObservaBCN} is function-based, all the other methods are state-based. 
			So, the method in  \cite{Li2015ControlObservaBCN} shows remarkably different efficiencies 
			when being applied to BCNs with different structures, but the other methods show similar efficiencies. 
			In \cite{Cheng2016NoteonObservabilityBCN,Zhu2018ObservabilityBCN}, the complexity is lower
			than that in \cite{Zhang2014ObservabilityofBCNCCC}, because in 
			\cite{Cheng2016NoteonObservabilityBCN,Zhu2018ObservabilityBCN} the authors jumped over 
			the procedure of computing at most $2^{2n}$ deterministic finite automata (each with complexity
			$O(2^{2n+m})$) 
			from a weighted pair graph adopted in \cite{Zhang2014ObservabilityofBCNCCC}, and directly used
			the observability matrix and the observability graph to verify observability. This idea was earlier
			used in \cite{Zhang2016WPGRepresentationReconBCN} when verifying reconstructibility of BCNs.}
		\label{tab1:ObsResulBCN}
	\end{table*}

\subsection{Potential applications}

The verification problem for controllability or observability of 
(infinite-state) hybrid systems is generally undecidable. If one can construct an LCN
as a finite abstraction that (bi)simulates a given hybrid system in the sense of preserving controllability
or observability, then one can verify controllability or observability for the hybrid system by verifying
the LCN. An attempt of using a similar scheme to verify opacity of (infinite-state) 
transition systems can be found in \cite{Zhang2019OpacitySimulationTS}. Related results on using 
finite abstractions to do verification or synthesis for infinite-state systems
can be found in \cite{zam14,Coogan2016FormalMethodsTrasfficFlow}, etc.

As for the synthesis problem, it is known that for linear control systems,
state feedback with exogenous input does not affect controllability,
but may affect observability \cite{Wonham1985LinearMultiControl}.
However, both properties may be affected by state feedback with exogenous input
for nonlinear control systems and hybrid systems. 
By using a simulation-based method, if one can construct an LCN as a finite abstraction
that (bi)simulates a given unobservable hybrid system in the sense of preserving 
observability, then one can first try to find a state-feedback controller to make the obtained unobservable LCN
observable, and then refine the obtained controller into the original hybrid system so as to make the 
original hybrid system observable. Here we do not mention controllability because 
state feedback with exogenous input cannot enforce controllability for LCNs but sometimes can
weaken their controllability \cite{Zhang2019CDCSynthesisBCN}: 
If an LCN is not controllable, then no state-feedback controller with exogenous input can make it controllable;
if an LCN is controllable, then there may exist a state-feedback controller with exogenous input that 
makes the LCN uncontrollable.

\subsection{Contribution}

The main contributions of this paper are as follows: Let $\Sig$ be an LCN and $\Sig_{\Ccal}$ an LCN
obtained by feeding a state-feedback controller $\Ccal$ with exogenous input into $\Sig$.

\begin{enumerate}
	\item\label{item4:synthesisConObser}
		We prove that state feedback with exogenous input sometimes can enforce observability of LCNs. 
		If $\Sig$ is observable, then $\Sig_{\Ccal}$ can be either observable or not;
		if $\Sig$ is unobservable, then $\Sig_{\Ccal}$ can also be either observable or not.
	\item\label{item5:synthesisConObser}
		We prove that if an unobservable $\Sig$ can be made observable by state feedback
		with exogenous input,
		then it can also be made observable by state feedback (without exogenous input,
		equivalent to state feedback with constant exogenous input). This result yields an
		algorithm for verifying whether an unobservable LCN can be made observable by 
		state feedback with exogenous input,
		since there are finitely many state-feedback controllers
		(although there are infinitely many state-feedback controllers with exogenous input).
	\item\label{item6:synthesisConObser}
		We also obtain an upper bound on the number of state-feedback controllers that are needed 
		to be substituted into an unobservable LCN $\Sig$ to check whether $\Sig$ can be made observable by 
		state feedback, and based on the procedure of obtaining the upper bound,
		we design an observability synthesis algorithm, by additionally combining the ideas of a greedy
		algorithm and dynamic programming, opening the study of observability synthesis in LCNs.
\end{enumerate}

The above 
\eqref{item4:synthesisConObser} had been presented at the
58th IEEE Conference on Decision and
Control 2019 \cite{Zhang2019CDCSynthesisBCN} and are also illustrated in Table~\ref{tab1:mainresult}.
The other contributions show substantially new results 
compared with 
\eqref{item4:synthesisConObser}.


\begin{table*}[!htbp]
	\centering
	\begin{tabular}[]{|c||c|c|c|c|}
		\hline\rowcolor{lightgray}
		& can enforce & can weaken & can enforce & can weaken\\
		\rowcolor{lightgray}
		& controllability? & controllability? & observability? & observability?\\\hline
		\begin{tabular}[]{c}
			state feedback\\ with exogenous input
		\end{tabular}
		&
		\begin{tabular}[]{c}
			No\\ (\cite{Zhang2019CDCSynthesisBCN})
		\end{tabular}
		& 
		\begin{tabular}[]{c}
			Yes\\ (\cite{Zhang2019CDCSynthesisBCN})
		\end{tabular}
		& 
		\begin{tabular}[]{c}
			Yes\\ (Exam.~\ref{exam4:synthesisConObser})
		\end{tabular}
		&
		\begin{tabular}[]{c}
			Yes\\ (Exam.~\ref{exam7:synthesisConObser})
		\end{tabular}
		\\\hline
	\end{tabular}
	\caption{Influence of state feedback with exogenous input to 
	controllability and observability of LCNs.}
	\label{tab1:mainresult}
\end{table*}

The remainder of this paper is organized as follows. Section~\ref{sec2:pre} introduces preliminaries,
i.e., LCNs with their algebraic form under the STP framework, basic verification methods
for observability of LCNs. The main results are shown in 
Section~\ref{sec4:mainresultObservability}.
Section~\ref{sec5:conc} is a short conclusion.

\section{Preliminary results}\label{sec2:pre}

\subsection{The semitensor product of matrices}\label{subsec:STP}

We introduce necessary notation as follows.

\begin{itemize}
	\item $\subset$ and $\subsetneq$ denote subset and strict subset relations, respectively
	\item $2^X$: power set of  set $X$
  \item $\Z_+$: set of  positive integers
  \item $\N$: set of  natural numbers (including $0$)
  \item $\R^n$: set of $n$-length real column vectors
  \item $\R^{m\times n}$: set of $m\times n$ real matrices
  \item $\D_k$: set $\{0,\frac{1}{k-1},...,1\}$ of $k$-value logic
  \item $\delta_n^i$: $i$-th column of the identity matrix $I_n$
    \item ${\bf 1}_k$: $\sum_{i=1}^k{\delta_k^i}$
  \item $\Delta_n$: set $\{\delta_n^1,\dots,\delta_n^n\}$
	  ($\Dt:=\Dt_2$)
 \item $\llb m,n\rrb$: $\{m,m+1,...,n\}$, where $m,n\in\N$ and $m\le n$
  \item $\delta_n[i_1,\dots,i_s]$:  logical matrix
	  $[\delta_n^{i_1},\dots,\delta_n^{i_s}]$, where $i_1,\dots,i_s\in\llb 1,n\rrb$
  \item $\LM_{n\times s}$:  set of
  $n\times s$ logical matrices 
  \item $\col_i(A)$: $i$-th column of matrix $A$
  \item $\col(A)$: set of columns of matrix $A$
  \item $A^T$: transpose of matrix $A$
  \item $|X|$:  cardinality of set $X$
   \item $A_1\oplus A_2\oplus\cdots\oplus A_n$: $\begin{bmatrix}
A_1 & 0 & \cdots & 0\\
0   & A_2 & \cdots & 0\\
\vdots & \vdots & \ddots & \vdots\\
0   & 0   & \cdots  & A_n
\end{bmatrix}$
\end{itemize}

\begin{definition}[\cite{Cheng2011book}]
	Let $A\in \R^{m\times n}$, $B\in \R^{p\times q}$, and $\alpha=\mbox{lcm}
	(n,p)$ be the least common multiple of $n$ and $p$. The STP of $A$
	and $B$ is defined as
	\begin{equation*}
		A\ltimes B = \left(A\otimes I_{\frac{\alpha}{n}}\right)\left(B\otimes I_{\frac
		{\alpha}{p}}\right),
		\label{def_of_stp}
	\end{equation*}
	where $\otimes$ denotes the Kronecker product.
\end{definition}

From this definition, it is easy to see that the conventional
product of matrices is a particular case of the STP, since if $n=p$ then $A\ltimes B=AB$.
Since the STP keeps most properties of the conventional product,
e.g., the associative law \cite{Zhang2020bookDDS}\footnote{In \cite{Cheng2011book}, the associative law
$(A\ltimes B)\ltimes C=A\ltimes(B\ltimes C)$ was proven in the special case that $n$ divides $p$ (or vice versa)
and $q$ divides $r$ (or vice versa), where $n$ and $q$ are the numbers of columns of $A$ and $B$, respectively,
$p$ and $r$ are the numbers of rows of $B$ and $C$, respectively.},
the distributive law, etc. \cite{Cheng2011book},
we usually omit the symbol ``$\ltimes$'' hereinafter.

\subsection{Logical control networks and their algebraic form}

In this paper, we investigate the following LCN
with $n$ state nodes, $m$ input nodes, and $q$ output nodes:
\begin{equation}\label{LCN1:synthesisConObser}
\begin{split}
 x_1 (t + 1) &= f_1 (x_1 (t), \dots ,x_n (t),u_1 (t), \dots ,u_m (t)), \\
 x_2 (t + 1) &= f_2 (x_1 (t), \dots ,x_n (t),u_1 (t), \dots ,u_m (t)), \\
  &\quad\!\!\vdots  \\
 x_n (t + 1) &= f_n (x_1 (t), \dots ,x_n (t),u_1 (t), \dots ,u_m (t)),\\
 y_1 (t ) &= h_1 (x_1 (t), \dots ,x_n (t)), \\
 y_2 (t ) &= h_2 (x_1 (t), \dots ,x_n (t)), \\
  &\quad\!\!\vdots  \\
 y_q (t ) &= h_q (x_1 (t), \dots ,x_n (t)),
 \end{split}
\end{equation}
where $t\in\N$ denote discrete time steps; $x_i(t)\in\D_{n_i}$, $u_j(t)\in\D_{m_j}$, 
and $y_k(t)\in\D_{q_k}$ denote 
values of state node $x_i$, input node $u_j$, and output node $y_k$ at time step $t$,
respectively,
$i\in\llb 1,n\rrb$, $j\in\llb 1,m\rrb$, $k\in\llb 1,q\rrb$; $\prod_{i=1}^n{n_i}=:N$; $\prod_{j=1}^m{m_i}=:M$;
$\prod_{k=1}^q{q_i}=:Q$;
$f_i:\D_{MN}\to\D_{n_i}$ and $h_k:\D_{N}\to \D_{q_k}$
are logical mappings, $i\in\llb 1,n\rrb$, $k\in\llb 1,q\rrb$.

When $n_1=\cdots=n_n=m_1=\cdots=m_m=q_1=\cdots =q_q=2$, Eqn.~\eqref{LCN1:synthesisConObser} is a BCN.

Eqn.~\eqref{LCN1:synthesisConObser} can be represented in the compact form
\begin{equation}\label{LCN2:synthesisConObser}
	\begin{split}
		x(t+1) &= f(x(t),u(t)),\\
		y(t) &= h(x(t)),
	\end{split}
\end{equation}
where $t\in\N$; $x(t)\in\D_{N}$, $u(t)\in\D_{M}$, and $y(t)\in\D_{Q}$
stand for the state, input, and output of the LCN at time step $t$;
$f:\D_{NM}\to \D_{N}$ and $h:\D_{N}\to \D_{Q}$ are mappings.

For each $n\in\Z_{+}$ greater than $1$, we map $\frac{n-i}{n-1}$ in $\D_n$ to $\delta_n^i$ 
in $\Dt_n$ for all $i\in\llb 1,n\rrb$, and write $\frac{n-i}{n-1}\sim\delta_n^i$. 
Then under the STP framework, Eqn.~\eqref{LCN2:synthesisConObser} can be transformed to its equivalent
algebraic form as follows \cite{Cheng2011book}:
\begin{equation}\label{LCN3:synthesisConObser}
	\begin{split}
		\tilde x (t + 1) &= L \tilde x (t) \tilde u(t)=[L_1,\dots,L_N] \tilde x (t) \tilde u(t),\\
		\tilde y (t ) &= H \tilde x (t), 
	\end{split}
\end{equation}
where $t\in\N$; $\tilde x(t)\in\Dt_{N}$, $\tilde u(t)\in\Dt_M$, $\tilde y(t)\in\Dt_Q$;
$L\in\LM_{N\times NM}$ and $H\in\LM_{Q\times N}$ are called the \emph{structure matrices},
$L_i\in\LM_{N\times M}$, $i\in\llb 1,N\rrb$.

\subsection{Preliminary results for observability}

In \cite{Zhang2014ObservabilityofBCNCCC,Zhang2016ObservabilityofBCN}, four types of observability were
verified for BCNs by proposing a unified automaton method (computing four types of deterministic
finite automata from the observability graph (proposed in 
\cite{Zhang2014ObservabilityofBCNCCC,Zhang2016ObservabilityofBCN} and
called weighted pair graph therein, and renamed observability graph in \cite{Zhang2020ObsLargeBCN,Zhang2020bookDDS}) 
of a BCN to verify the corresponding four types of 
observability). In this paper, we are particularly interested in the linear type
(as in Definition~\ref{def1:observability_aggregation}, i.e., the strongest one among the four types,
earlier studied in \cite{Fornasini2013ObservabilityReconstructibilityofBCN}),
as if an LCN satisfies this observability property, it is very easy to recover the initial state by using
an input sequence and the corresponding output sequence. Note that all results in \cite{Zhang2014ObservabilityofBCNCCC,Zhang2016ObservabilityofBCN}
can be trivially extended to LCNs. The necessary and sufficient condition for 
Definition~\ref{def1:observability_aggregation} given in \cite{Fornasini2013ObservabilityReconstructibilityofBCN}
is as follows: a BCN satisfies Definition~\ref{def1:observability_aggregation} if and only if for every 
pair of different periodic (state, input)-trajectories of the same minimal period $k$ and the same input trajectory,
the corresponding output trajectories are also different and periodic of minimal period $k$.  Obviously, 
the necessary and sufficient condition given in \cite{Fornasini2013ObservabilityReconstructibilityofBCN} is much 
more complex than the one given in the subsequent Lemma~\ref{lem1:observability_aggregation} \cite{Zhang2016ObservabilityofBCN}.

The four types of observability studied in \cite{Zhang2014ObservabilityofBCNCCC,Zhang2016ObservabilityofBCN} are 
as follows (also see Table~\ref{tab1:ObsResulBCN}). We adopt the terminology used in \cite{Zhang2020bookDDS}.

\begin{definition}\label{def1:observability_aggregation}
	An LCN~\eqref{LCN2:synthesisConObser} is called \emph{arbitrary-experiment observable} if
	for all different initial states $x(0),x'(0)\in\D_{N}$, for each input sequence
	$u(0)u(1)\dots$, the corresponding output sequences 
	$y(0)y(1)\dots$ and $y'(0)y'(1)\dots$ are different. 
\end{definition}

\begin{definition}\label{def4_observability}
	An LCN~\eqref{LCN2:synthesisConObser} is called \emph{multiple-experiment observable}
	if for very two different initial states
	$x(0),x(0)'\in\D_N$, there is an input sequence such that the output sequences corresponding to
	$x(0)$ and $x(0)'$ are different. Such an input sequence is called a \emph{distinguishing input
	sequence} of $x(0)$ and $x(0)'$.
\end{definition}

\begin{definition}\label{def1_observability}
	An LCN~\eqref{LCN2:synthesisConObser} is called \emph{strongly multiple-experiment observable}
	if for every initial state $x(0)\in\D_N$, there exists an input sequence such that
	for each initial state $x(0)'\in\D_N$ different from $x(0)$, the output sequences corresponding to
	$x(0)$ and $x(0)'$ are different.
\end{definition}

\begin{definition}\label{def7_observability}
	An LCN~\eqref{LCN2:synthesisConObser} is called \emph{single-experiment observable}
	if there exists an input sequence such that for every two different initial states $x(0),x(0)'\in\D_N$,
	the output sequences corresponding to $x(0)$ and $x(0)'$ are different.
\end{definition}

From now on when we mention ``observability'', we always mean Definition~\ref{def1:observability_aggregation}
unless otherwise stated.

Now we introduce the notion of observability graph.

\begin{definition}\cite{Zhang2014ObservabilityofBCNCCC,Zhang2016ObservabilityofBCN}\label{def2:observability_aggregation}
	Consider an LCN~\eqref{LCN2:synthesisConObser}.
	A triple $\Gr_o=(\V,\E,\W)$ is called
	its \emph{observability graph} if $\V$ (elements of $\V$ are called vertices) is equal to
	$ \{\{x,x'\}\in\D_N \times\D_N|h(x)=h(x')\}$\footnote{vertices are unordered state 
	pairs, i.e., $\{x,x'\}=\{x',x\}$.}, 
	$\E$ (elements of $\E$ are called edges) is equal to
	$\{(\{x_1,x_1'\},\{x_2,x_2'\})\in\V\times\V|
	\text{there exists }u\in\D_M\text{ such that }f(x_1,u)=x_2\text{ and }f(x_1',u)=x_2',
	\text{ or, }f(x_1,u)=x_2'\text{ and }f(x_1',u)=x_2\}\subset \V\times\V$, and
	the weight function $\W:\E\to 2^{\D_M}$ assigns to each edge $(\{x_1,x_1'\},\{x_2,x_2'\})\in\E$ a set $
	\{u\in\D_M|f(x_1,u)=x_2\text{ and }f(x_1',u)=x_2',
	\text{ or, }f(x_1,u)=x_2'\text{ and }f(x_1',u)=x_2	\}$ of inputs. 
	A vertex $\{x,x'\}$ is called \emph{diagonal} if $x=x'$, and called \emph{non-diagonal} otherwise.
	For a vertex $v\in\V$, its \emph{outdegree} is $\outdeg(v):=|\bigcup_{(v,v')\in\E}\W((v,v'))|$, i.e.,
	the number of inputs appearing in the edges starting from $v$. 
	The \emph{diagonal subgraph} of an observability graph is defined by
	all diagonal vertices and all edges between them.
	Similarly, the \emph{non-diagonal subgraph} is defined by
	all non-diagonal vertices and all edges between them.
\end{definition}

\begin{lemma}[\cite{Zhang2016ObservabilityofBCN}]\label{lem1:observability_aggregation}
	An LCN~\eqref{LCN2:synthesisConObser} is not observable if and only if
	in its observability graph $\Gr_o$ there is a non-diagonal vertex $v$, a cycle $C$, and
	a path from $v$ to a vertex of $C$. Particularly if in $\Gr_o$ there exists a path from $v$ to a
	diagonal vertex, then \eqref{LCN2:synthesisConObser} is not observable.
\end{lemma}

Since in a diagonal subgraph, there must exist a cycle and each vertex will go to a cycle,
we will denote the subgraph briefly by a symbol 
$\diamond$ when drawing an observability graph.
Hence if there exists an edge from a non-diagonal vertex to a diagonal vertex,
then the LCN is not observable.

\begin{example}\label{exam2:synthesisConObser}
	Consider the BCN
	\begin{equation}\label{eqn9:synthesisConObser}
		\tilde x(t+1)=L\tilde x(t)\tilde u(t),
	\end{equation}
	where $L=\delta_4[2,  2,  1,  3,  4,  4,  2,  2]$,
	$t\in\N$, $\tilde x(t)\in\Dt_4$, $\tilde u(t)\in\Dt$.
	Consider the output function
	\begin{equation}\label{eqn11:synthesisConObser}
		\tilde y(t)=\delta_2[1,1,1,2]\tilde x(t),
	\end{equation}
	where $t\in\N$, $\tilde x(t)\in\Dt_4$, $\tilde y(t)\in\Dt$.
	The observability graph of BCN~\eqref{eqn9:synthesisConObser} with output 
	function~\eqref{eqn11:synthesisConObser} is shown in Fig.~\ref{fig1:synthesisConObser}.
	This graph shows that the BCN is not observable by Lemma~\ref{lem1:observability_aggregation}
	since there is a self-loop on non-diagonal vertex $\{\dt_4^1,\dt_4^2\}$.

	\begin{figure}[!htbp]
        \centering
\begin{tikzpicture}[shorten >=1pt,auto,node distance=1.5 cm, scale = 1.0, transform shape,
	>=stealth,inner sep=2pt,state/.style={
	rectangle,minimum size=6mm,rounded corners=3mm,
	very thick,draw=black!50,
	top color=white,bottom color=black!50,font=\ttfamily},
	point/.style={rectangle,inner sep=0pt,minimum size=2pt,fill=}]
	\node[state] (12)                                 {$12$};
	\node[state] (13) [right of = 12]               {$13$};
	\node[state] (23) [below of = 12]               {$23$};
	\node[state] (di) [right of = 23]               {$\diamond$};

	\path [->] (12) edge [loop left] node {$1$} (12)
	      [->] (12) edge node {$2$} (23)
		  [->] (di) edge [loop right] node {$1,2$} (di)
		  ;
        \end{tikzpicture}
		\caption{Observability graph of BCN~\eqref{eqn9:synthesisConObser} with output 
		function~\eqref{eqn11:synthesisConObser},
		where number $ij$ in a circle denotes state pair $\{\delta_4^i,\delta_4^j\}$,
		weight $i$ denotes input $\delta_2^i$.}
		\label{fig1:synthesisConObser}
	\end{figure}
\end{example}

\section{Synthesis results for observability}\label{sec4:mainresultObservability}

In this section, we show synthesis results for observability of 
LCN~\eqref{LCN2:synthesisConObser} (or its algebraic form \eqref{LCN3:synthesisConObser})
based on state feedback with exogenous input.

\subsection{Closed-loop logical control networks by state feedback with exogenous input}

Consider an LCN~\eqref{LCN2:synthesisConObser}. Let a \emph{state-feedback controller
with exogenous input} be
\begin{equation}\label{eqn2:synthesisConObser}
	u(t)=g(x(t),v(t)),
\end{equation}
where $v(t)\in\D_P$ is the exogenous input, $P=\prod_{l=1}^p{p_l}$ with each $p_l\in\N$ greater than $1$
(corresponding to $l$ new input nodes); or $P=1$, which means that
there is only one constant input; $g:\D_{NP}\to\D_M$ is a mapping.
Equivalently in the algebraic form, for an LCN~\eqref{LCN3:synthesisConObser},
we set a state-feedback controller with exogenous input to be 
\begin{equation}\label{eqn3:synthesisConObser}
\tilde u(t)=G\tilde x(t)\tilde v(t)=[G_1,\dots,G_N]\tilde x(t)\tilde v(t),
\end{equation}
where $\tilde v(t)\in\Dt_P$, $G\in\LM_{M\times NP}$ is called the structure matrix,
$G_i\in\LM_{M\times P}$, $i\in\llb 1,N\rrb$.

In particular, when $P=1$,
a state-feedback controller with exogenous input is called a \emph{state-feedback controller}.
When $P=M$ and $g(x(t),v(t))\equiv v(t)$, controller~\eqref{eqn3:synthesisConObser} 
will not change the algebraic 
form of the original LCN,
and hence will not change controllability or observability of the LCN.

Suppose we are free to modify LCN~\eqref{LCN2:synthesisConObser} by setting 
controller~\eqref{eqn2:synthesisConObser}.
Substituting \eqref{eqn2:synthesisConObser} into \eqref{LCN2:synthesisConObser}, we obtain
a closed-loop LCN (see Fig.~\ref{fig8:synthesisConObser} for a sketch) as
\begin{equation}\label{eqn4:synthesisConObser}
	\begin{split}
		x(t+1) &= f(x(t),g(x(t),v(t))),\\
		y(t) &= h(x(t)).
	\end{split}
\end{equation}
\begin{figure}[!htbp]
	\begin{center}
		\begin{tikzpicture}[auto,>=latex',
	block/.style={
  draw, 
  draw=black!50,top color=white,bottom color=black!50,
  rectangle, 
  minimum height=2em, 
  minimum width=4em
  },
sum/.style={
  draw, 
  draw=black!50,top color=white,bottom color=black!50,
  circle, 
  },
input/.style={coordinate},
output/.style={coordinate},
pinstyle/.style={
  pin edge={to-,thin,black}
  }]
    \node [input, name=input] {};
    \node [sum, right = of input] (sum) {};
    \node [block, right = of sum] (controller) {Controller};
    \node [block, right = of controller, 
            node distance=3cm] (system) {System};
    \draw [->] (controller) -- node[name=u] {$u$} (system);
    \node [output, right =of system] (output) {};

    \draw [draw,->] (input) -- node {$v$} (sum);
    \draw [->] (sum) -- node {} (controller);
    \draw [->] (system) -- node [name=y] {$x$}(output);
    \draw [->] (y) -- ++(0,-1.5cm) -| 
        node [near end] {} (sum);
\end{tikzpicture}
	\end{center}
	\caption{Closed-loop logical control network based on state feedback with exogenous input.}
	\label{fig8:synthesisConObser}
\end{figure} 

Equivalently, substituting \eqref{eqn3:synthesisConObser} into \eqref{LCN3:synthesisConObser}, we obtain
the algebraic form of the closed-loop LCN~\eqref{eqn4:synthesisConObser} as
\begin{equation}\label{eqn5:synthesisConObser}
\begin{split}
 \tilde x (t + 1) &= L \tilde x (t) G \tilde x(t) \tilde v(t),\\
 \tilde y (t ) &= H \tilde x (t).
 \end{split}
\end{equation}

\begin{proposition}\label{prop1:synthesisConObser}
	Eqn~\eqref{eqn5:synthesisConObser} is equivalent to
	\begin{equation}\label{eqn6:synthesisConObser}
		\begin{split}
			 \tilde x (t + 1) &= [L_1G_1,...,L_NG_N] \tilde x (t) \tilde v(t),\\
			 \tilde y (t ) &= H \tilde x (t).
		 \end{split}
	\end{equation}
\end{proposition}

\begin{proof}
	By Lemmas~\ref{prop4:STP} and ~\ref{prop2:STP}, \eqref{eqn5:synthesisConObser} can be rewritten as
	\begin{align*}
		\tilde x(t+1) &= L\tilde x (t) G \tilde x(t) \tilde v(t)\\
		&= L(I_N\otimes G)M_{N_r}\tilde x(t)v(t)\\
		&= L
		\begin{bmatrix}
			G\\&\ddots\\&&G
		\end{bmatrix}
		\left( 
		\begin{bmatrix}
			\delta_N^1\\&\ddots\\&&\delta_N^N
		\end{bmatrix}
		\otimes I_P\right)
		\tilde x(t)\tilde v(t)\\
		&= L
		\begin{bmatrix}
			G\\&\ddots\\&&G
		\end{bmatrix}
		\begin{bmatrix}
			\delta_N^1\otimes I_P\\&\ddots\\&&\delta_N^N\otimes I_P
		\end{bmatrix}
		\tilde x(t)\tilde v(t)\\
		&= L
		\begin{bmatrix}
			G(\delta_N^1\otimes I_P)\\&\ddots\\&&G(\delta_N^N\otimes I_P)
		\end{bmatrix}
		\tilde x(t)\tilde v(t)\\
		&= [L_1,\dots,L_N]\begin{bmatrix}
			G_1\\&\ddots\\&&G_N
		\end{bmatrix}
		\tilde x(t)\tilde v(t)\\
		&= [L_1G_1,...,L_NG_N] \tilde x (t) \tilde v(t).
	\end{align*}
\end{proof}

Consider the newly obtained LCN~\eqref{eqn6:synthesisConObser}, if $P=1$, 
then the
corresponding structure matrix $[L_1G_1,\dots,L_NG_N]$ is square.
However, generally the structure matrix is not necessarily
square, hence the updating of states generally depends on the exogenous input $\tilde v(t)$.

\subsection{How state feedback influences observability of LCNs}

Unlike controllability, we next give an example to show that state feedback with exogenous input
can enforce observability
of an LCN.
\begin{example}\label{exam4:synthesisConObser}
	We have proved that BCN~\eqref{eqn9:synthesisConObser} with the output function~\eqref{eqn11:synthesisConObser}
	is not observable in Example~\ref{exam2:synthesisConObser}.
	
	Substituting the controller
	\begin{equation}\label{eqn12:synthesisConObser}
		\tilde u(t)=G\tilde x(t)\tilde v(t),
	\end{equation}
	where $G=\delta_2[1,2,2,2,1,2,1,2]$, $\tilde v(t)\in\Dt$, $\tilde x(t)\in\Dt_4$,
	into \eqref{eqn9:synthesisConObser}
	to obtain the closed-loop BCN
	\begin{equation}\label{eqn10:synthesisConObser}
		\tilde x(t+1) = \tilde L{\blue\tilde x(t)}\tilde u(t),
	\end{equation}
	where $\tilde L=\delta_4[2,  2,  3,  3,  4,  4,  2,  2]$.
	The observability graph of BCN~\eqref{eqn10:synthesisConObser}
	with output function~\eqref{eqn11:synthesisConObser} is shown in 
	Fig.~\ref{fig2:synthesisConObser}. This graph shows that the BCN is observable by
	Lemma~\ref{lem1:observability_aggregation}.
	\begin{figure}[!htbp]
        \centering
\begin{tikzpicture}[shorten >=1pt,auto,node distance=1.5 cm, scale = 1.0, transform shape,
	>=stealth,inner sep=2pt,state/.style={
	rectangle,minimum size=6mm,rounded corners=3mm,
	very thick,draw=black!50,
	top color=white,bottom color=black!50,font=\ttfamily},
	point/.style={rectangle,inner sep=0pt,minimum size=2pt,fill=}]
	\node[state] (12)                                 {$12$};
	\node[state] (13) [right of = 12]               {$13$};
	\node[state] (23) [below of = 12]               {$23$};
	\node[state] (di) [right of = 23]               {$\diamond$};

	\path 
	      [->] (12) edge node {$1,2$} (23)
		  [->] (di) edge [loop right] node {$1,2$} (di)
		  ;
        \end{tikzpicture}
		\caption{Observability graph of BCN~\eqref{eqn10:synthesisConObser} with output 
		function~\eqref{eqn11:synthesisConObser}.}
	\label{fig2:synthesisConObser}
	\end{figure}

\end{example}


Next we show that a state-feedback controller can make an observable LCN unobservable.
\begin{example}\label{exam7:synthesisConObser}
	Consider the LCN
	\begin{equation}\label{eqn30:synthesisConObser}
		\begin{split}
			\tilde x(t+1) &= \delta_3[1,  3,  3,  2,  1,  1]\tilde x(t)\tilde u(t),\\
			\tilde y(t) &= \delta_2[1,1,2]\tilde x(t),
		\end{split}
	\end{equation}
	where $t\in\N$, $\tilde x(t)\in\Dt_3$, $\tilde u(t),\tilde y(t)\in\Dt$.

	The observability graph of \eqref{eqn30:synthesisConObser} consists of 
	vertex $\{\delta_3^1,\delta_3^2\}$ and the diagonal subgraph $\diamond$,
	and there is no path from $\{\delta_3^1,\delta_3^2\}$ to $\diamond$.
	Then by Lemma~\ref{lem1:observability_aggregation}, the BCN is observable.
%
%

	Substituting state-feedback controller $$\tilde u(t)=\delta_2[1,2,1]\tilde x(t)$$ into \eqref{eqn30:synthesisConObser},
	by Proposition~\ref{prop1:synthesisConObser}, we obtain LCN
	\begin{equation}\label{eqn31:synthesisConObser}
		\begin{split}
			\tilde x(t+1) &= \delta_3[1,2,1]\tilde x(t),\\
			\tilde y(t) &= \delta_2[1,1,2]\tilde x(t).
		\end{split}
	\end{equation}
	There is a self-loop on vertex $\{\delta_3^1,\delta_3^2\}$ in the observability graph 
	of \eqref{eqn31:synthesisConObser}, then by Lemma~\ref{lem1:observability_aggregation},
	\eqref{eqn31:synthesisConObser} is not observable.
\end{example}

Next we show that there exists an unobservable LCN such that no state-feedback controller 
with exogenous input can make it observable. This example also shows that sometimes
state feedback with exogenous input never affects observability of LCNs.

\begin{example}\label{exam5:synthesisConObser}
	Consider the BCN
	\begin{equation}\label{eqn13:synthesisConObser}
		\tilde x(t+1)=L\tilde x(t)\tilde u(t),
	\end{equation}
	where $L=\delta_4[1,  1,  1,  1,  1,  1,  2,  3]$,
	$t\in\N$, $\tilde x(t)\in\Dt_4$, $\tilde u(t)\in\Dt$.

	By Lemma~\ref{lem1:observability_aggregation}, the BCN with output
	function~\eqref{eqn11:synthesisConObser} is not observable, since there exists a path
	$\{\delta_4^1,\delta_4^2\}\xrightarrow[]{\delta_2^1}\{\delta_4^1,\delta_4^1\}
	\xrightarrow[]{\delta_2^1}\{\delta_4^1,\delta_4^1\}$ in its observability graph.

	Substituting an arbitrary state-feedback controller $\tilde u(t)=G\tilde x(t)\tilde v(t)$ with exogenous input 
	$\tilde v(t)$, where $G\in\LM_{2\times 4P}$, $\tilde v(t)\in\Dt_P$,
	$P$ is an arbitrary positive integer, into \eqref{eqn13:synthesisConObser},
	by Proposition~\ref{prop1:synthesisConObser}, we obtain closed-loop LCN
	\begin{equation}\label{eqn14:synthesisConObser}
		\tilde x(t+1)=\left[ \delta_4^1\otimes{\bf 1}_P^T,\delta_4^1\otimes{\bf 1}_P^T,\delta_4^1\otimes{\bf 1}_P^T,
		L_4G_4\right] \tilde x(t)\tilde u(t),
	\end{equation}
	where $L_4G_4=\delta_4[i_1,\dots,i_P]$, $i_1,\dots,i_P\in\llb 2,3\rrb$.

	The observability graph of \eqref{eqn14:synthesisConObser} with output function~\eqref{eqn11:synthesisConObser}
	contains a path $\{\delta_4^1,\delta_4^2\}\xrightarrow[]{\delta_P^1}
	\{\delta_4^1,\delta_4^1\}\xrightarrow[]{\delta_P^1}\{\delta_4^1,\delta_4^1\}$, then the LCN
	is not observable by Lemma~\ref{lem1:observability_aggregation}.
\end{example}

Based on the above discussion, we know that state feedback with exogenous input
sometimes can enforce observability of
an LCN, sometimes cannot. Next we study when a state-feedback controller can enforce observability.

\subsection{Controller synthesis for enforcing observability of LCNs}
\label{subsec:ContSynObs}

The following main result shows that in order to test whether an unobservable LCN can be made
observable by state feedback with exogenous input, it is enough to check whether the LCN can be 
made observable by state feedback.

\begin{theorem}\label{thm2:synthesisConObser}
	Consider an unobservable LCN~\eqref{LCN3:synthesisConObser}. If it can be made observable
	by a state-feedback controller~\eqref{eqn3:synthesisConObser} with exogenous input,
	then it can also be made observable
	by a state-feedback controller (i.e., \eqref{eqn3:synthesisConObser} with $P=1$).
\end{theorem}

\begin{proof}
	Assume an unobservable LCN~\eqref{LCN3:synthesisConObser} and a controller~\eqref{eqn3:synthesisConObser} 
	that makes \eqref{LCN3:synthesisConObser} observable.
	Then by Lemma~\ref{lem1:observability_aggregation}, in the observability graph $\Gr_o$
	of the corresponding closed-loop LCN $\Sig$ obtained by  
	substituting \eqref{eqn3:synthesisConObser} into \eqref{LCN3:synthesisConObser}, 
	\begin{equation}\label{eqn16:synthesisConObser}
		\begin{split}
			&\text{there exists no cycle in its non-diagonal subgraph,}\\
			&\text{and there exists no edge} \\
			&\text{from any non-diagonal vertex to any diagonal vertex.}
		\end{split}
	\end{equation}

	Now consider the structure matrix $G=[G_1,\dots,G_N]$ of \eqref{eqn3:synthesisConObser},
	we choose a new state-feedback controller
	\begin{equation}\label{eqn15:synthesisConObser}
		\tilde u(t) = \left[\col_i(G_1),\dots,\col_i(G_N)\right]\tilde x(t),
	\end{equation}
	where $i\in\llb 1,P\rrb$ is arbitrarily given,
	and consider the observability graph  $\Gr_o'$ of the closed-loop LCN $\Sig'$ obtained by
	substituting \eqref{eqn15:synthesisConObser} into \eqref{LCN3:synthesisConObser}.

	It can be seen that the vertex sets of $\Gr_o$ and $\Gr_o'$ coincide, since $\Sig$ and $\Sig'$ have the same 
	output function. One also sees that for every two vertices $v$ and $v'$ in the vertex set,
	if there exists an edge from $v$ to $v'$ in $\Gr_o'$, then there also exists an edge from $v$ to $v'$ in $\Gr_o$,
	i.e., the edge set of $\Gr_o'$ is a subset of that of $\Gr_o$.
	Hence $\Gr_o'$ also satisfies \eqref{eqn16:synthesisConObser}, and $\Sig'$ is also observable
	by Lemma~\ref{lem1:observability_aggregation}. 
\end{proof}


Because there are infinitely many state-feedback controllers with exogenous input, 
generally one cannot directly check whether an unobservable LCN can be made observable 
by state feedback with exogenous input. 
However, by Theorem~\ref{thm2:synthesisConObser},
one can do the above check because there are totally finitely many
state-feedback controllers. Formally, the following Theorem~\ref{thm3:synthesisConObser} holds.

\begin{theorem}\label{thm3:synthesisConObser}
	An unobservable LCN~\eqref{LCN3:synthesisConObser} can be made observable
	by state feedback with exogenous input if and only if it can be made observable
	by state feedback.
\end{theorem}

By Proposition~\ref{prop1:synthesisConObser} and the proof of Theorem~\ref{thm2:synthesisConObser}, the
following result holds. The subsequent discussions on observability synthesis will be based on this result.
\begin{theorem}\label{thm7:synthesisConObser}
	An LCN~\eqref{LCN3:synthesisConObser} can be made observable by state feedback with exogenous input
	if and only if there exist $i_1,\dots,i_N\in\llb 1,M \rrb$ such that the BN
	\begin{equation}\label{LCN6:synthesisConObser} 
	\begin{split}
		\tilde x (t + 1) &= [\col_{i_1}(L_1),\dots,\col_{i_N}(L_N)] \tilde x (t) ,\\
		\tilde y (t ) &= H \tilde x (t),
	\end{split}
\end{equation}
is observable.
\end{theorem}

\begin{remark}\label{rem3:synthesisConObser}
	In order to verify whether an unobservable LCN~\eqref{LCN3:synthesisConObser} can be made 
	observable by state feedback, one should substitute several state-feedback controllers into the LCN,
	and then check whether there exists an observable closed-loop LCN.
	Now we analyze how many state-feedback controllers
	should be substituted into the LCN in order to do the verification.

Consider LCN~\eqref{LCN3:synthesisConObser}, it is sufficient to substitute
$\prod_{i=1}^N|\col(L_i)|$ state-feedback controllers into the LCN to do the above verification.
It is because, in order not to do repetitive check,
for every two of the chosen state-feedback controllers,
where their structure matrices are $G_1=[g^1_1,\dots,g^1_N]$ and $G_2=[g^2_1,\dots,g^2_N]$, respectively,
they must satisfy (see \eqref{LCN6:synthesisConObser})
\begin{equation}\label{eqn17:synthesisConObser}
	\left[L_1g^1_1,\dots,L_Ng^1_N\right]\ne\left[L_1g^2_1,\dots,L_Ng^2_N\right],
\end{equation}
which means the obtained closed-loop LCNs are different.
On the other hand, in order not to lose any necessary check, it is sufficient to choose
$\prod_{i=1}^N|\col(L_i)|$ state-feedback
controllers every two of which satisfy \eqref{eqn17:synthesisConObser} to do the above check.
\end{remark}

Next we give some special conditions for whether an unobservable LCN can be made
observable by state feedback, which can be checked under much less computational cost than the equivalent condition
in Theorem~\ref{thm7:synthesisConObser}. In addition, using these conditions we can furthermore reduce 
the number (shown in Remark~\ref{rem3:synthesisConObser})
of state-feedback controllers that are needed to be substituted into the unobservable LCN to do the above check.

\begin{theorem}\label{thm4:synthesisConObser}
	Consider an unobservable LCN~\eqref{LCN3:synthesisConObser}. The LCN cannot be made
	observable by any state-feedback controller if at least one of the following holds.
	\begin{myenumerate}
		\item\label{item1:synthesisConObser}
			$L$ satisfies $L_j=L_k=\delta_N^l\otimes{\bf1}_M^T$ for some different $j,k\in\llb 1,N\rrb$ 
			and some $l\in\llb 1,N\rrb$, where $H\delta_N^j=H\delta_N^k$.
		\item\label{item2:synthesisConObser}
			There exist different $j,k\in\llb 1,N\rrb$ such that $L_j=\delta_N^j\otimes{\bf1}_M^T$ and 
			$L_k=\delta_N^k\otimes{\bf1}_M^T$, or $L_j=\delta_N^k\otimes{\bf1}_M^T$ and 
			$L_k=\delta_N^j\otimes{\bf1}_M^T$, where $H\delta_N^j=H\delta_N^k$.
	\end{myenumerate}
\end{theorem}

\begin{proof} 
	Assume \eqref{item1:synthesisConObser} holds, then in the observability graph of the closed-loop LCN
	obtained by feeding an arbitrary state-feedback controller into the original unobservable LCN,
	there exists an edge 
	$\{\delta_N^j,\delta_N^k\}\xrightarrow[]{\delta_M^1}\{\delta_N^l,\delta_N^l\}$, where $\{\delta_N^l,
	\delta_N^l\}$ is a diagonal vertex. Then by Lemma~\ref{lem1:observability_aggregation}, no obtained
	closed-loop LCN is observable.

	Assume \eqref{item2:synthesisConObser} holds,  then in the observability graph of the closed-loop LCN
	obtained by feeding an arbitrary state-feedback controller into the original unobservable LCN,
	there exists a self-loop $\{\delta_N^j,\delta_N^k\}\xrightarrow[]{\delta_M^1}\{\delta_N^j,\delta_N^k\}$
	on the non-diagonal vertex $\{\delta_N^j,\delta_N^k\}$,
	then also by Lemma~\ref{lem1:observability_aggregation}, no obtained closed-loop LCN is observable.
\end{proof}

By Theorem~\ref{thm4:synthesisConObser} and the previously obtained results, we show how to furthermore reduce
the number (shown in Remark~\ref{rem3:synthesisConObser}) of state-feedback controllers that are needed
to be substituted into an unobservable LCN to check whether the LCN can be made observable by state feedback.
Consider an unobservable LCN~\eqref{LCN3:synthesisConObser} and a state-feedback controller
\begin{equation}\label{eqn18:synthesisConObser}
	\tilde u(t) = G\tilde x(t) = [g_1,\dots,g_N]\tilde x(t),
\end{equation}
where $g_i\in\LM_{M\times 1}$, $i\in\llb 1,N\rrb$.
Substituting \eqref{eqn18:synthesisConObser} into \eqref{LCN3:synthesisConObser}, we obtain a closed-loop LCN
\begin{equation}\label{eqn19:synthesisConObser}
\begin{split}
 \tilde x (t + 1) &= [L_1g_1,\dots,L_Ng_N] \tilde x (t) ,\\
 \tilde y (t ) &= H \tilde x (t) 
 \end{split}
\end{equation}
by Proposition~\ref{prop1:synthesisConObser}, which is consistent with Eqn.~\eqref{LCN6:synthesisConObser}.

Denote 
\begin{equation}\label{eqn20:synthesisConObser}
	\col(H)=\left\{ \delta_Q^{k_1},\dots,\delta_Q^{k_\ell} \right\},
\end{equation}
where $\delta_Q^{k_1},\dots,\delta_Q^{k_\ell}$
are distinct. For each $i\in\llb 1,l\rrb$, we denote 
\begin{equation}\label{eqn21:synthesisConObser}
	\begin{split}
	S_{k_i}:=&\left\{\left.\delta_N^j\right|j\in\llb 1,N\rrb,H\delta_N^j=
	\delta_Q^{k_i}\right\},\\
	c_i :=& \left|S_{k_i}\right|,\\
	S_{k_i} =:& \left\{\delta_N^{i_1},\dots,\delta_N^{i_{c_i}}\right\}.
	\end{split}
\end{equation}
The collection of state sets $S_{k_1},\dots,S_{k_\ell}$ partitions $\Dt_N$.

In order to make \eqref{eqn19:synthesisConObser} observable, we must assume that 
for each $i\in\llb 1,\ell\rrb$, for all different $j,k\in\llb 1,N\rrb$ with $\delta_N^j,\delta_N^k\in S_{k_i}$, 
it holds that $L_jg_j\ne L_kg_k$. Otherwise, in the observability graph of \eqref{eqn19:synthesisConObser},
there exists an edge $\{\delta_N^j,\delta_N^k\}\to\{L_jg_j, L_kg_k\}$, 
where $\{\delta_N^j,\delta_N^k\}$ is a non-diagonal vertex, and $\{L_jg_j, L_kg_k\}$ is a diagonal vertex,
which shows that \eqref{eqn19:synthesisConObser} is not observable by Lemma~\ref{lem1:observability_aggregation}.
Hence in order to make \eqref{eqn19:synthesisConObser} observable, we must furthermore assume that
\begin{equation}\label{eqn38:synthesisConObser} 
	\begin{split}
		&\text{for each }i\in\llb 1,\ell\rrb,\text{ it holds that}\\
		&|\{L_jg_j|j\in\llb 1,N\rrb,\delta_N^j\in S_{k_i}\}| = |S_{k_i}|=c_i,
	\end{split}
\end{equation}
i.e., $L_jg_j$ with $\dt_N^j\in S_{k_i}$ are distinct.  

The above analysis yields the following result stronger than the first implication
of Theorem~\ref{thm4:synthesisConObser} (i.e., \eqref{item1:synthesisConObser} implies that
\eqref{LCN3:synthesisConObser} cannot be made observable by any state-feedback controller).

\begin{theorem}\label{thm5:synthesisConObser}
	Given an unobservable LCN~\eqref{LCN3:synthesisConObser}, assume there exists $i\in\llb 1,\ell\rrb$
	($\ell$ is defined in \eqref{eqn20:synthesisConObser})
	such that for all $j_1,\dots,j_{c_i}\in\llb 1,M\rrb$ ($c_i$ is defined in \eqref{eqn21:synthesisConObser}),
	$|\{\col_{j_1}(L_{i_1}),\dots,\col_{j_{c_i}}(L_{i_{c_i}})\}|<c_i$, then \eqref{LCN3:synthesisConObser} cannot
	be made observable by any state-feedback controller.
\end{theorem}

\begin{proof}
	Consider an arbitrary given state-feedback controller~\eqref{eqn18:synthesisConObser}
	and the corresponding closed-loop LCN~\eqref{eqn19:synthesisConObser} obtained by
	substituting \eqref{eqn18:synthesisConObser} into
	the unobservable \eqref{LCN3:synthesisConObser}.
	By assumption, there exists $i\in\llb 1,\ell\rrb$ and different $\alpha,\beta\in\llb 1,N\rrb$ such that
	$\delta_N^{\alpha},\delta_N^{\beta}\in S_{k_i}$ and $L_{\alpha}g_{\alpha}=
	L_{\beta}g_{\beta}$. Hence in the observability graph of 
	\eqref{eqn19:synthesisConObser}, there exists an edge $\{\delta_N^{\alpha},\delta_N^{\beta}\}
	\to \{L_{\alpha}g_{\alpha},L_{\beta}g_{\beta}\}$,
	where $\{\delta_N^{\alpha},\delta_N^{\beta}\}$ is a non-diagonal vertex, 
	$\{L_{\alpha}g_{\alpha},L_{\beta}g_{\beta}\}$ is a diagonal vertex. By 
	Lemma~\ref{lem1:observability_aggregation}, \eqref{eqn19:synthesisConObser} is not observable.
\end{proof}

\begin{remark}\label{rem4:synthesisConObser}
	Theorem~\ref{thm5:synthesisConObser} is stronger than the first implication in
	Theorem~\ref{thm4:synthesisConObser}, because \eqref{item1:synthesisConObser} in
	Theorem~\ref{thm4:synthesisConObser} is stronger the assumption in Theorem~\ref{thm5:synthesisConObser},
	but they have the same conclusion
	(i.e., \eqref{LCN3:synthesisConObser} cannot be made observable by any state-feedback controller).
\end{remark}

Based on these analysis, the following result holds.
\begin{theorem}\label{thm6:synthesisConObser}
	Consider an unobservable LCN~\eqref{LCN3:synthesisConObser}.
	In order to verify whether \eqref{LCN3:synthesisConObser} can be made observable by state feedback,
	it is sufficient to substitute 
	\begin{equation}\label{eqn22:synthesisConObser}
		\prod_{i=1}^{\ell}Num_i	
	\end{equation} 
	state-feedback controllers of the form \eqref{eqn18:synthesisConObser} into
	\eqref{LCN3:synthesisConObser} to check whether there exists an observable closed-loop LCN,
	where \begin{equation}\label{eqn29:synthesisConObser}\begin{split}
		Num_i=\left|\left\{ \left.(\alpha_{i_1},\dots,\alpha_{i_{c_i}})\right|\right.\right.&
		\alpha_{i_k}\in\col(L_{i_k}),k\in\llb 1,c_i\rrb,\\
		&\alpha_{i_1},\dots,\alpha_{i_{c_i}}\text{ are distinct}
		\}|,\end{split}\end{equation}
	$\ell$ is defined in \eqref{eqn20:synthesisConObser}, 
	$c_i$ and $i_1,\dots,i_{c_i}$ are defined in \eqref{eqn21:synthesisConObser}.
	In addition, for every two of the above \eqref{eqn22:synthesisConObser}
	chosen state-feedback controllers,
	their structure matrices $G_1=[g^1_1,\dots,g^1_N]$ and $G_2=[g^2_1,\dots,g^2_N]$
	must satisfy  
	\begin{equation}\label{eqn23:synthesisConObser}
		\left[L_{j_1}g^1_{j_1},\dots,L_{j_{c_j}}g^1_{j_{c_j}}\right]\ne
		\left[L_{j_1}g^2_{j_1},\dots,L_{j_{c_j}}g^2_{j_{c_j}}\right]
	\end{equation} 
	for some $j\in\llb 1,\ell\rrb$. Particularly, if \eqref{eqn22:synthesisConObser} is equal to $0$,
	then the unobservable 
	LCN~\eqref{LCN3:synthesisConObser} cannot be made observable by state feedback.
\end{theorem}

\begin{proof}
	Firstly, observe that if \eqref{eqn22:synthesisConObser} is equal to $0$, then the assumption in 
	Theorem~\ref{thm5:synthesisConObser} is satisfied, and then the unobservable LCN~\eqref{LCN3:synthesisConObser}
	cannot be made observable by state feedback.

	On the contrary, if the assumption in Theorem~\ref{thm5:synthesisConObser} is not satisfied,
	then for all $i\in\llb 1,\ell\rrb$,
	there exist $j_1,\dots,j_{c_i}\in\llb 1,M\rrb$ such that
	$$\left|\left\{\col_{j_1}(L_{i_1}),\dots,\col_{j_{c_i}}(L_{i_{c_i}})\right\}\right|=c_i,$$
	which implies that $Num_i>0$.
	Hence \eqref{eqn22:synthesisConObser} is greater than $0$.

	Secondly, if \eqref{eqn22:synthesisConObser} is greater than $0$,
	we can find state-feedback controllers such that the observability graphs of the obtained closed-loop
	LCNs must not have an edge from a non-diagonal vertex to a diagonal vertex. We check whether some of 
	these closed-loop LCNs is observable. In order not to do repetitive check,
	for every two of the found state-feedback controllers, where their structure matrices are denoted by
	$G_1=[g^1_1,\dots,g^1_N]$ and $G_2=[g^2_1,\dots,g^2_N]$, respectively,
	they must satisfy that there exists $i\in\llb 1,\ell\rrb$ such that \eqref{eqn23:synthesisConObser} holds,
	which implies that the obtained two closed-loop LCNs are different.
	On the other hand, in order not to lose any necessary check, it is sufficient to choose 
	\eqref{eqn22:synthesisConObser} state-feedback controllers every two of which satisfy
	\eqref{eqn23:synthesisConObser} to do the above check. 
\end{proof}

\begin{remark}\label{rem5:synthesisConObser}
It is easy to see that \eqref{eqn22:synthesisConObser} is no greater than the number
$\prod_{i=1}^N|\col(L_i)|$
shown in Remark~\ref{rem3:synthesisConObser}, hence Theorem~\ref{thm6:synthesisConObser}
strengthens the result shown in Remark~\ref{rem3:synthesisConObser}.

Note that the above number \eqref{eqn22:synthesisConObser} is obtained by avoiding the existence
of an edge from a non-diagonal vertex to a diagonal vertex in the observability graph
of an obtained closed-loop LCN. In addition to avoiding this, we must also avoid the existence 
of cycles in the non-diagonal subgraphs of the observability graphs,
so the minimal number of state-feedback controllers  
that are needed to do the above check could be further reduced.
\end{remark}



In the remainder of this section, we give an observability synthesis algorithm based on 
Lemma~\ref{lem1:observability_aggregation} and the results obtained in Section~\ref{subsec:ContSynObs}.
To this end, we first give a motivating example.

\subsection{A motivating example}

\begin{example}\label{exam6:synthesisConObser}
	Consider the following LCN
	\begin{equation}\label{eqn24:synthesisConObser}
		\begin{split}
			\tilde x(t+1) = \delta_8[&1,1,2,3,2,3,1,4,3,5,7,6,6,7,8,1,\\
			&2,3,7,6,1,2,3,4,3,4,7,8,5,6,7,4]\\
			&\tilde x(t)\tilde u(t),\\
			\tilde y(t) = \delta_4[&1,1,1,1,1,2,2,2]\tilde x(t),
		\end{split}
	\end{equation}
	where $t\in\N$, $\tilde x(t)\in\Dt_8$, $\tilde u(t),\tilde y(t)\in\Delta_4$. In the observability graph
	of \eqref{eqn24:synthesisConObser}, there exists a path 
	$$\left\{ \delta_8^2,\delta_8^5 \right\}\xrightarrow[]{\delta_4^2}
	\left\{ \delta_8^3,\delta_8^3 \right\}
	\xrightarrow[]{\delta_4^1}\left\{ \delta_8^3,\delta_8^3 \right\},$$ hence
	\eqref{eqn24:synthesisConObser} is not observable by Lemma~\ref{lem1:observability_aggregation}.

	\begin{figure}[!htbp]
        \centering
\begin{tikzpicture}[shorten >=1pt,auto,node distance=1.5 cm, scale = 1.0, transform shape,
	>=stealth,inner sep=2pt,state/.style={
	rectangle,minimum size=6mm,rounded corners=3mm,
	very thick,draw=black!50,
	top color=white,bottom color=black!50,font=\ttfamily},
	point/.style={rectangle,inner sep=0pt,minimum size=2pt,fill=}]
	\node[state] (68)                               {$68$};
	\node[state] (78) [right of = 68]               {$78$};
	\node[state] (35) [right of = 78]               {$35$};
	\node[state] (67) [right of = 35]               {$67$};
	\node[state] (15) [below of = 68]               {$15$};
	\node[state] (12) [right of = 15]               {$12$};
	\node[state] (23) [right of = 12]               {$23$};
	\node[state] (13) [right of = 23]               {$13$};
	\node[state] (14) [below of = 15]               {$14$};
	\node[state] (24) [right of = 14]               {$24$};
	\node[state] (25) [right of = 24]               {$25$};
	\node[state] (di) [right of = 25]               {$\diamond$};
	\node[state] (34) [below of = 14]               {$34$};
	\node[state] (45) [right of = 34]               {$45$};


	\path [->] (12) edge [loop right] (12)
		  [->] (13) edge [loop left] (13)
		  [->] (23) edge [loop left] (23)
	      [->] (68) edge (15)
		  [->] (15) edge (12)
		  [->] (78) edge (35) 
		  [->] (35) edge (23)
		  [->] (67) edge (13)
		  [->] (25) edge (di)
		  [->] (di) edge [loop right] (di)
		  ;
        \end{tikzpicture}
		\caption{Observability graph of LCN~\eqref{eqn26:synthesisConObser}.}
		\label{fig3:synthesisConObser}
	\end{figure}

	Next we try to find a state-feedback controller (if any)
	\begin{equation}\label{eqn25:synthesisConObser}
		\tilde u(t)=\delta_4[i_1,\dots,i_8]\tilde x(t)
	\end{equation}to 
	make \eqref{eqn24:synthesisConObser} observable, where $i_1,\dots,i_8\in\llb 1,4\rrb$.

	We might as well firstly choose $i_1=\cdots=i_8=1$. Substituting this controller 
	into \eqref{eqn24:synthesisConObser}, we obtain the closed-loop LCN
	\begin{equation}\label{eqn26:synthesisConObser}
		\begin{split}
			\tilde x(t+1) &= \delta_8[1,2,3,6,2,1,3,5]\tilde x(t),\\
			\tilde y(t) &= \delta_4[1,1,1,1,1,2,2,2]\tilde x(t).
		\end{split}
	\end{equation}
	The observability graph of \eqref{eqn26:synthesisConObser} (shown in Fig.~\ref{fig3:synthesisConObser}) 
	contains self-loops in its non-diagonal subgraph,
	hence \eqref{eqn26:synthesisConObser} is not observable by Lemma~\ref{lem1:observability_aggregation}.

	Secondly, we try to modify $i_1,\dots,i_8$ to make \eqref{eqn24:synthesisConObser} observable.
	The basic idea guiding us to choose new $i_1,\dots,i_8$ is to remove all self-loops 
	and all edges from a non-diagonal vertex to a diagonal vertex in Fig.~\ref{fig3:synthesisConObser}.
	Since there exists a self-loop on vertex $\{\delta_8^1,\delta_8^2\}$ originally,
	we keep $i_1=1$ invariant, and change $i_2$ from $1$ to $2$, then the self-loop on $\{\delta_8^1,\delta_8^2\}$
	is changed to an edge $\{\delta_8^1,\delta_8^2\}\to\{\delta_8^1,\delta_8^3\}$.
	Since there also exists a self-loop on vertex $\{\delta_8^1,\delta_8^3\}$ originally,
	we change $i_3$ to $2$, then the self-loop on $\{\delta_8^1,\delta_8^3\}$
	is changed to an edge $\{\delta_8^1,\delta_8^3\}\to\{\delta_8^1,\delta_8^5\}$. 
	Now consider vertex $\{\delta_8^1,\delta_8^5\}$. Since originally $\{\delta_8^1,\delta_8^5\}$ goes to the 
	self-loop on $\{\delta_8^1,\delta_8^2\}$, we change $i_5$ from $1$ to $3$, then there exists no edge
	from vertex $\{\delta_8^1,\delta_8^5\}$ to any vertex. After doing these modifications,
	there exists no edge from vertex $\{\delta_8^2,\delta_8^5\}$ to any vertex
	(originally there is a path from $\{\delta_8^2,\delta_8^5\}$ to the diagonal subgraph),
	there is a path from $\{\delta_8^2,\delta_8^3\}$ to $\{\delta_8^3,\delta_8^5\}$, and 
	there exists no edge from $\{\delta_8^3,\delta_8^5\}$ to any vertex
	(originally there is a self-loop on $\{\delta_8^2,\delta_8^3\}$).
	
	Now we substitute the new state-feedback controller
	\begin{equation}\label{eqn28:synthesisConObser}
		\tilde u(t)=\delta_4[1,2,2,1,3,1,1,1]\tilde x(t)
	\end{equation}
	into \eqref{eqn24:synthesisConObser}, we then obtain a new closed-loop LCN
	\begin{equation}\label{eqn27:synthesisConObser}
		\begin{split}
			\tilde x(t+1) &= \delta_8[1,3,5,6,7,1,3,5]\tilde x(t),\\
			\tilde y(t) &= \delta_4[1,1,1,1,1,2,2,2]\tilde x(t).
		\end{split}
	\end{equation}
	Luckily, in the observability graph of \eqref{eqn27:synthesisConObser} (shown in 
	Fig.~\ref{fig4:synthesisConObser}), there is no path from any non-diagonal vertex to any cycle,
	hence \eqref{eqn27:synthesisConObser} is observable by Lemma~\ref{lem1:observability_aggregation}.
	Hence LCN~\eqref{eqn24:synthesisConObser} can be made observable by 
	state-feedback controller~\eqref{eqn28:synthesisConObser}.

	Now we compute the upper bounds on the numbers of state-feedback controllers that
	are needed to be tested in order to verify whether LCN~\eqref{eqn24:synthesisConObser} can be made
	observable by state feedback obtained respectively in Remark~\ref{rem3:synthesisConObser} and
	Theorem~\ref{thm6:synthesisConObser}.

	For \eqref{eqn24:synthesisConObser}, we have $$\col(H)=\left\{ \delta_4^1,\delta_4^2 \right\},$$
	where we denote $k_1=1$ and $k_2=2$. Then
	\begin{align*}
		S_{k_1} &= \left\{ \delta_8^1,\delta_8^2,\delta_8^3,\delta_8^4,\delta_8^5 \right\},
		\quad c_1=\left|S_{k_1}\right|=5,\\
		S_{k_2} &= \left\{ \delta_8^6,\delta_8^7,\delta_8^8 \right\},\quad c_2=\left|S_{k_2}\right|=3,
	\end{align*}
	$$\begin{array}{ll}
		\col(L_1)=\left\{ \delta_8^1,\delta_8^2,\delta_8^3 \right\},
		&\col(L_2)=\left\{ \delta_8^1,\delta_8^2,\delta_8^3,\delta_8^4 \right\},\\
		\col(L_3)=\left\{\delta_8^3,\delta_8^5,\delta_8^6,\delta_8^7 \right\}, 
		&\col(L_4)=\left\{\delta_8^1,\delta_8^6,\delta_8^7,\delta_8^8 \right\},\\
		\col(L_5)=\left\{\delta_8^2,\delta_8^3,\delta_8^6,\delta_8^7 \right\},
		&\col(L_6)=\left\{\delta_8^1,\delta_8^2,\delta_8^3,\delta_8^4 \right\},\\
		\col(L_7)=\left\{\delta_8^3,\delta_8^4,\delta_8^7,\delta_8^8 \right\}, 
		&\col(L_8)=\left\{\delta_8^4,\delta_8^5,\delta_8^6,\delta_8^7 \right\}.
	\end{array}$$

	As shown in Remark~\ref{rem3:synthesisConObser}, the upper bound is 
	$$\prod_{i=1}^8\left|\col(L_i)\right|=3\cdot 4^7=49152.$$

	As shown in Theorem~\ref{thm6:synthesisConObser}, 
	the corresponding upper bound \eqref{eqn22:synthesisConObser} is equal to $$Num_1\cdot Num_2=153\cdot
	46=7038,$$
	where $Num_i$'s are defined by \eqref{eqn29:synthesisConObser}.

	\begin{figure}[!htbp]
        \centering
\begin{tikzpicture}[shorten >=1pt,auto,node distance=1.5 cm, scale = 1.0, transform shape,
	>=stealth,inner sep=2pt,state/.style={
	rectangle,minimum size=6mm,rounded corners=3mm,
	very thick,draw=black!50,
	top color=white,bottom color=black!50,font=\ttfamily},
	point/.style={rectangle,inner sep=0pt,minimum size=2pt,fill=}]
	\node[state] (68)                               {$68$};
	\node[state] (78) [right of = 68]               {$78$};
	\node[state] (35) [right of = 78]               {$35$};
	\node[state] (67) [right of = 35]               {$67$};
	\node[state] (15) [below of = 68]               {$15$};
	\node[state] (12) [right of = 15]               {$12$};
	\node[state] (23) [right of = 12]               {$23$};
	\node[state] (13) [right of = 23]               {$13$};
	\node[state] (14) [below of = 15]               {$14$};
	\node[state] (24) [right of = 14]               {$24$};
	\node[state] (25) [right of = 24]               {$25$};
	\node[state] (34) [below of = 14]               {$34$};
	\node[state] (45) [right of = 34]               {$45$};
	\node[state] (di) [right of = 25]               {$\diamond$};


	\path 
	      [->] (12) edge [bend left] (13)
		  [->] (13) edge [bend left] (15)
		  [->] (23) edge (35) 
		  [->] (68) edge (15)
		  [->] (67) edge (13)
		  [->] (78) edge (35)
		  [->] (45) edge (67)
		  [->] (di) edge [loop right] (di)
		  ;
        \end{tikzpicture}
		\caption{Observability graph of LCN~\eqref{eqn27:synthesisConObser}.}
		\label{fig4:synthesisConObser}
	\end{figure}
	
	Next we give a relatively ``smarter'' method to made LCN~\eqref{eqn24:synthesisConObser}
	observable. We first choose state-feedback controller
	\begin{equation}\label{eqn32:synthesisConObser}
		\tilde u(t)=\delta_4[1,1,1,1,3,1,1,1]\tilde x(t).
	\end{equation}
	After substituting \eqref{eqn32:synthesisConObser} into \eqref{eqn24:synthesisConObser},
	we obtain the following closed-loop LCN
	\begin{equation}\label{eqn33:synthesisConObser}
		\begin{split}
			\tilde x(t+1) &= \delta_8[1,2,3,6,7,1,3,5]\tilde x(t),\\
			\tilde y(t) &= \delta_4[1,1,1,1,1,2,2,2]\tilde x(t).
		\end{split}
	\end{equation}
	The purpose of choosing \eqref{eqn32:synthesisConObser} is to force the observability 
	graph of the obtained closed-loop LCN~\eqref{eqn33:synthesisConObser} to have no edge from a non-diagonal vertex
	to a diagonal vertex.
	The observability graph of \eqref{eqn33:synthesisConObser} is shown in Fig.~\ref{fig5:synthesisConObser}.
	\begin{figure}[!htbp]
        \centering
\begin{tikzpicture}[shorten >=1pt,auto,node distance=1.5 cm, scale = 1.0, transform shape,
	>=stealth,inner sep=2pt,state/.style={
	rectangle,minimum size=6mm,rounded corners=3mm,
	very thick,draw=black!50,
	top color=white,bottom color=black!50,font=\ttfamily},
	point/.style={rectangle,inner sep=0pt,minimum size=2pt,fill=}]
	\node[state] (68)                               {$68$};
	\node[state] (78) [right of = 68]               {$78$};
	\node[state] (35) [right of = 78]               {$35$};
	\node[state] (67) [right of = 35]               {$67$};
	\node[state] (15) [below of = 68]               {$15$};
	\node[state] (12) [right of = 15]               {$12$};
	\node[state] (23) [right of = 12]               {$23$};
	\node[state] (13) [right of = 23]               {$13$};
	\node[state] (14) [below of = 15]               {$14$};
	\node[state] (24) [right of = 14]               {$24$};
	\node[state] (25) [right of = 24]               {$25$};
	\node[state] (34) [below of = 14]               {$34$};
	\node[state] (45) [right of = 34]               {$45$};
	\node[state] (di) [right of = 25]               {$\diamond$};


	\path [->] (12) edge [loop right] (12)
		  [->] (13) edge [loop left] (13)
		  [->] (23) edge [loop left] (23)
	      [->] (68) edge (15)
		  [->] (78) edge (35) 
		  [->] (67) edge (13)
		  [->] (45) edge (67)
		  [->] (di) edge [loop right] (di)
		  ;
        \end{tikzpicture}
		\caption{Observability graph of LCN~\eqref{eqn33:synthesisConObser}.}
		\label{fig5:synthesisConObser}
	\end{figure}
	There are three self-loops on three non-diagonal vertices in Fig.~\ref{fig5:synthesisConObser}.
	The next step is to remove these three
	self-loops and meanwhile disable appearance of edges from a non-diagonal vertex to a 
	diagonal vertex. Consider the self-loop on vertex $\{\dt_8^1,\dt_8^2\}$. We change
	\eqref{eqn32:synthesisConObser} to the following
	\begin{equation}\label{eqn34:synthesisConObser}
		\tilde u(t)=\delta_4[1,4,1,1,3,1,1,1]\tilde x(t).
	\end{equation} After substituting \eqref{eqn34:synthesisConObser} into \eqref{eqn24:synthesisConObser},
	we obtain the following closed-loop LCN
	\begin{equation}\label{eqn35:synthesisConObser}
		\begin{split}
			\tilde x(t+1) &= \delta_8[1,4,3,6,7,1,3,5]\tilde x(t),\\
			\tilde y(t) &= \delta_4[1,1,1,1,1,2,2,2]\tilde x(t).
		\end{split}
	\end{equation}
	In the observability graph (Fig.~\ref{fig6:synthesisConObser}) of \eqref{eqn35:synthesisConObser},
	there is still no edge from a non-diagonal vertex to a diagonal vertex, but only 
	one of the previous three self-loops on non-diagonal vertices is left.
	\begin{figure}[!htbp]
        \centering
\begin{tikzpicture}[shorten >=1pt,auto,node distance=1.5 cm, scale = 1.0, transform shape,
	>=stealth,inner sep=2pt,state/.style={
	rectangle,minimum size=6mm,rounded corners=3mm,
	very thick,draw=black!50,
	top color=white,bottom color=black!50,font=\ttfamily},
	point/.style={rectangle,inner sep=0pt,minimum size=2pt,fill=}]
	\node[state] (68)                               {$68$};
	\node[state] (78) [right of = 68]               {$78$};
	\node[state] (35) [right of = 78]               {$35$};
	\node[state] (67) [right of = 35]               {$67$};
	\node[state] (15) [below of = 68]               {$15$};
	\node[state] (12) [right of = 15]               {$12$};
	\node[state] (23) [right of = 12]               {$23$};
	\node[state] (13) [right of = 23]               {$13$};
	\node[state] (14) [below of = 15]               {$14$};
	\node[state] (24) [right of = 14]               {$24$};
	\node[state] (25) [right of = 24]               {$25$};
	\node[state] (34) [below of = 14]               {$34$};
	\node[state] (45) [right of = 34]               {$45$};
	\node[state] (di) [right of = 25]               {$\diamond$};


	\path [->] (12) edge (14)
		  [->] (13) edge [loop left] (13)
	      [->] (68) edge (15)
		  [->] (78) edge (35) 
		  [->] (67) edge (13)
		  [->] (45) edge (67)
		  [->] (di) edge [loop right] (di)
	      [->] [bend right] (23) edge (34)
		  ;
        \end{tikzpicture}
		\caption{Observability graph of LCN~\eqref{eqn35:synthesisConObser}.}
		\label{fig6:synthesisConObser}
	\end{figure}
	We furthermore change \eqref{eqn34:synthesisConObser} to the following
	\begin{equation}\label{eqn36:synthesisConObser}
		\tilde u(t)=\delta_4[1,4,2,1,3,1,1,1]\tilde x(t).
	\end{equation} After substituting \eqref{eqn36:synthesisConObser} into 
	into \eqref{eqn24:synthesisConObser},
	we obtain the following closed-loop LCN
	\begin{equation}\label{eqn37:synthesisConObser}
		\begin{split}
			\tilde x(t+1) &= \delta_8[1,4,5,6,7,1,3,5]\tilde x(t),\\
			\tilde y(t) &= \delta_4[1,1,1,1,1,2,2,2]\tilde x(t).
		\end{split}
	\end{equation}
	In the observability graph (Fig.~\ref{fig7:synthesisConObser}) of \eqref{eqn37:synthesisConObser},
	there is no cycle in the non-diagonal subgraph,
	LCN~\eqref{eqn37:synthesisConObser} is observable.
	\begin{figure}[!htbp]
        \centering
\begin{tikzpicture}[shorten >=1pt,auto,node distance=1.5 cm, scale = 1.0, transform shape,
	>=stealth,inner sep=2pt,state/.style={
	rectangle,minimum size=6mm,rounded corners=3mm,
	very thick,draw=black!50,
	top color=white,bottom color=black!50,font=\ttfamily},
	point/.style={rectangle,inner sep=0pt,minimum size=2pt,fill=}]
	\node[state] (68)                               {$68$};
	\node[state] (78) [right of = 68]               {$78$};
	\node[state] (35) [right of = 78]               {$35$};
	\node[state] (67) [right of = 35]               {$67$};
	\node[state] (15) [below of = 68]               {$15$};
	\node[state] (12) [right of = 15]               {$12$};
	\node[state] (23) [right of = 12]               {$23$};
	\node[state] (13) [right of = 23]               {$13$};
	\node[state] (14) [below of = 15]               {$14$};
	\node[state] (24) [right of = 14]               {$24$};
	\node[state] (25) [right of = 24]               {$25$};
	\node[state] (34) [below of = 14]               {$34$};
	\node[state] (45) [right of = 34]               {$45$};
	\node[state] (di) [right of = 25]               {$\diamond$};


	\path [->] (12) edge (14)
		  [->] (13) edge [bend right] (15)
		  [->] (23) edge (45)
	      [->] (68) edge (15)
		  [->] (78) edge (35) 
		  [->] (67) edge (13)
		  [->] (45) edge (67)
		  [->] (di) edge [loop right] (di)
		  ;
        \end{tikzpicture}
		\caption{Observability graph of LCN~\eqref{eqn37:synthesisConObser}.}
		\label{fig7:synthesisConObser}
	\end{figure}
\end{example}

In the above example, in the first method,
we are lucky that the second chosen state-feedback controller~\eqref{eqn28:synthesisConObser} makes
LCN~\eqref{eqn24:synthesisConObser} observable. If we modify $i_1,\dots,i_8$ in another way, 
the second chosen controller may not make \eqref{eqn24:synthesisConObser} observable. However,
due to the existence of \eqref{eqn28:synthesisConObser}, we know that finally we will find a
state-feedback controller (which might not be \eqref{eqn28:synthesisConObser})
that makes \eqref{eqn24:synthesisConObser} observable.

In the second method, when updating the state-feedback controllers (from 
\eqref{eqn32:synthesisConObser} to \eqref{eqn34:synthesisConObser}, then to
\eqref{eqn36:synthesisConObser}), we always have in the observability graphs of
the obtained closed-loop LCNs (from \eqref{eqn33:synthesisConObser} to
\eqref{eqn35:synthesisConObser}, then to \eqref{eqn37:synthesisConObser}),
there is no edge from a non-diagonal vertex
to a diagonal vertex, which is guaranteed by forcing the first $5$ columns of the structure
matrices of these closed-loop LCNs to be distinct and also forcing the last
$3$ columns to be distinct (i.e., \eqref{eqn38:synthesisConObser}). In this sense,
the obtained LCNs are closer and closer
to be observable, and it is very likely that after several steps, the obtained closed-loop
LCN is observable. Hence this method can result in an observability synthesis algorithm
that makes an unobservable LCN become observable (if possible) in a large extent, and the speed is not very slow.

\subsection{An observability synthesis algorithm for LCNs}

The above procedure of modifying state-feedback controllers is similar to the fundamental idea of a
greedy algorithm, i.e., in every step of modification, the purpose is to remove at least one (simple) cycle in
the non-diagonal subgraph of the corresponding observability graph, so that after several steps,
no cycle exists and the obtained closed-loop LCN is observable. Following the procedure, for a given
unobservable LCN~\eqref{LCN3:synthesisConObser} that can be made observable by state feedback, if the
initial state-feedback controller was chosen appropriately, and the following modifications were done 
appropriately, then the procedure could return a state-feedback controller that makes the unobservable
LCN observable. However, because of the high nonlinearity of the observability synthesis problem,
there is no guarantee that such an appropriate initial state-feedback controller could be definitely chosen,
neither for appropriate following steps of modifications. Hence in order to implement the procedure, one
could additionally add the fundamental idea of dynamic programming, i.e., rollback is permitted when in some step,
no modification can reduce the number of cycles. To sum up the hybrid procedure containing the ideas of a 
greedy algorithm and dynamic programming, we show the following Algorithm~\ref{alg1:synthesisConObser} for
observability synthesis.

\vspace{1em}

\begin{breakablealgorithm}
	\vspace{-1.00em}
	\caption{An observability synthesis algorithm}
	\label{alg1:synthesisConObser}
	\begin{algorithmic}[1]
		\REQUIRE An unobservable LCN $\Sig$ as in \eqref{LCN3:synthesisConObser}
		\ENSURE ``Yes'' if $\Sig$ can be made observable by state feedback, ``No'' otherwise;
		in case of ``Yes'', a state-feedback controller as in \eqref{eqn18:synthesisConObser} that makes
		$\Sig$ observable 
		\init\label{item7:synthesisConObser}
		A state-feedback controller $\Ccal$ as in \eqref{eqn18:synthesisConObser} such that the 
		closed-loop LCN $\Sig_{\Ccal}$ as in \eqref{eqn19:synthesisConObser} (obtained by feeding $\Ccal$ into $\Sig$)
		satisfies \eqref{eqn38:synthesisConObser} (this implies that the observability graph of 
		$\Sig_{\Ccal}$ contains no edge from any non-diagonal vertex to any diagonal vertex), and
		a threshold $1\le\alpha\le N$
		\IF{such an initial controller $\Ccal$ does not exist}
		\RETURN ``No''
		\STOP
		\ELSE 
		\IF{$\Sig_{\Ccal}$ is observable (i.e., the non-diagonal subgraph of its observability graph 
		contains no cycle by Lemma~\ref{lem1:observability_aggregation})}
		\RETURN ``Yes'' and $\Ccal$
		\STOP 
		\ELSE
		\WHILE{the current $\Sig_{\Ccal}$ is not observable, and, not all closed-loop LCNs as in 
		\eqref{eqn19:synthesisConObser} satisfying \eqref{eqn38:synthesisConObser} have been tested} 
		\STATE\label{item8:synthesisConObser}
		Choose to modify a number at most $\alpha$ of columns of the structure matrix of $\Ccal$ so that the number 
		of cycles in the non-diagonal subgraph of the observability graph of the updated $\Sig_{\Ccal}$ 
		decreases, $\Sig_{\Ccal}$ has not been tested, and meanwhile $\Sig_{\Ccal}$ still satisfies 
		\eqref{eqn38:synthesisConObser}
		\IF{such a modification does not exist}
		\STATE\label{item9:synthesisConObser}
		move backward to some of the previously tested controllers until such a modification exists
		and then do the modification as in Line~\ref{item8:synthesisConObser} (if after moving to the
		initial state-feedback controller but such a modification still does not exist, then reinitialize the
		procedure as in Line~\ref{item7:synthesisConObser})
		\ENDIF
		\IF{the current $\Sig_{\Ccal}$ is observable}
		\RETURN ``Yes'' and $\Ccal$
		\STOP
		\ENDIF
		\ENDWHILE
		\RETURN ``No''
		\STOP
		\ENDIF
		\ENDIF
	\end{algorithmic}
\end{breakablealgorithm}

Now we analyze Algorithm~\ref{alg1:synthesisConObser}. In Line~\ref{item7:synthesisConObser}, 
by Lemma~\ref{lem1:observability_aggregation}, if $\Sig_{\Ccal}$ does not satisfy 
\eqref{eqn38:synthesisConObser}, then it is not observable. Hence throughout Algorithm~\ref{alg1:synthesisConObser},
all closed-loop LCNs $\Sig_{\Ccal}$ must satisfy \eqref{eqn38:synthesisConObser}. Line~\ref{item8:synthesisConObser}
shows a modification method similar to the idea of a greedy algorithm, i.e., the obtained closed-loop LCN is
closer to be observable than the previous one, because observable closed-loop LCNs $\Sig_{\Ccal}$ are exactly
those containing no cycle in the non-diagonal subgraphs of their observability graphs by 
Lemma~\ref{lem1:observability_aggregation}. Line~\ref{item9:synthesisConObser} is a rollback similar to the 
idea of dynamic programming. It works when the modification in Line~\ref{item8:synthesisConObser} does not
work and not all closed-loop LCNs $\Sig_{\Ccal}$ as in \eqref{eqn19:synthesisConObser} satisfying 
\eqref{eqn38:synthesisConObser} have been tested. Particularly if $\alpha=N$, rollback is not needed.
The nondeterminism of Algorithm~\ref{alg1:synthesisConObser}
comes from Line~\ref{item7:synthesisConObser} (nondeterministically choosing an initial state-feedback controller),
Line~\ref{item8:synthesisConObser} (nondeterministically choosing columns of the structure matrix of the current
controller $\Ccal$), and Line~\ref{item9:synthesisConObser}. Overall, Algorithm~\ref{alg1:synthesisConObser} 
generates a tree structure, in which each node is a pair of a state-feedback controller $\Ccal$ and the corresponding
closed-loop LCN $\Sig_{\Ccal}$ satisfying \eqref{eqn38:synthesisConObser}. In each edge, the number of cycles
in the non-diagonal subgraph of the observability graph of the LCN in the head is smaller than that number 
corresponding to the LCN in the tail.

Algorithm~\ref{alg1:synthesisConObser} could be refined in several ways to change its running performance,
e.g., by adding additional rules, e.g., choosing $\alpha=1$,
setting priorities for columns of the structure matrix of $\Ccal$ to be chosen, etc.

\begin{example}\label{exam8:synthesisConObser}
	Recall the LCN~\eqref{eqn24:synthesisConObser} in Example~\ref{exam6:synthesisConObser}.
	The second method follows the procedure in Algorithm~\ref{alg1:synthesisConObser}, see Table~\ref{tab1:example}.
	\begin{table*}
		\centering
		\begin{tabular}{|c||c|c|c|}
			\hline
			\rowcolor{lightgray}
			STEP & controller $\Ccal$ & closed-loop LCN $\Sig_{\Ccal}$ &
			\begin{tabular}[]{c}
				number of cycles in the non-diagonal subgraph\\
				of the observability graph of $\Sig_{\Ccal}$
			\end{tabular}
			\\\hline
			1 & \eqref{eqn32:synthesisConObser} & \eqref{eqn33:synthesisConObser} & 3\\\hline
			2 & \eqref{eqn34:synthesisConObser} & \eqref{eqn35:synthesisConObser} & 1\\\hline
			3 & \eqref{eqn36:synthesisConObser} & \eqref{eqn37:synthesisConObser} & 0\\\hline
		\end{tabular}
		\caption{The second method in Example~\ref{exam6:synthesisConObser} for synthesizing a state-feedback
		controller that makes LCN~\eqref{eqn24:synthesisConObser} observable which follows the procedure shown in 
		Algorithm~\ref{alg1:synthesisConObser} (rollback is not needed in this example), where all closed-loop
		LCNs \eqref{eqn33:synthesisConObser}, \eqref{eqn35:synthesisConObser}, and \eqref{eqn37:synthesisConObser}
		satisfy \eqref{eqn38:synthesisConObser}, i.e., the non-diagonal subgraphs of their observability graphs
		contain no edge from a non-diagonal vertex to a diagonal vertex. \eqref{eqn37:synthesisConObser} is 
		observable.}
		\label{tab1:example}
	\end{table*}
\end{example}

\section{Conclusion}\label{sec5:conc}

In this paper, we showed that state feedback with exogenous input sometimes can enforce
{\blue or weaken} observability
of a logical control network (LCN). 
We also characterized how to verify whether observability of an LCN can be enforced by
state feedback with exogenous input. In addition, we gave an upper bound on the number of state-feedback
controllers that are needed to be tested in order to verify whether an unobservable LCN can be made
observable by state feedback with exogenous input. Finally, based on the method of obtaining the 
upper bound, an observability synthesis algorithm
was designed by additionally combining the ideas of a greedy algorithm and dynamic programming.
Note that the observability synthesis algorithm is preliminary, a lot of work could be done 
to improve its running performance. 

In this paper, we only studied the synthesis problem for observability of LCNs in the sense of 
Definition~\ref{def1:observability_aggregation}. The other three types of observability 
as in Definitions~\ref{def4_observability}, \ref{def1_observability}, and \ref{def7_observability}
can also be studied by using our observability graph, where in order to study Definitions~\ref{def1_observability} 
and \ref{def7_observability}, one additionally needs to compute deterministic finite automata from
an observability graph as introduced in Section~\ref{subsec:LiterRev}.

In order to make the obtained results be applied to the simulation-based method for controller synthesis 
of hybrid systems over their finite abstractions introduced in the Introduction section,
a future topic is to generalize the obtained results 
to nondeterministic finite-transition systems, since usually nondeterministic finite-transition systems
better simulate hybrid systems.
In addition, another natural generalization of the paper is to consider output-feedback controllers
instead of state-feedback controllers.

\appendix

\section{Concepts and properties related to the STP of matrices}

\begin{definition}
	The matrix $M_{k_r}={\dt}_k^1\oplus\cdots\oplus{\dt}_k^k$ is called
	the \emph{power-reducing matrix}\index{power-reducing matrix}.
	Particularly, we denote $M_{2_r}:=M_r$.
\end{definition}

By definition, the following result holds.

\begin{lemma}[\cite{Cheng2011book}]\label{prop4:STP}
	For power-reducing matrix $M_{k_r}$, we have $$P^2=M_{k_r}P$$ for each $P\in\Delta_k$.
\end{lemma}

\begin{lemma}[\cite{Cheng2011book}]\label{prop2:STP}
  Let $A\in{\mathbb R}^{m\times n}$ and $z\in{\mathbb R}^t$. Then
  \begin{align*}
    A\ltimes z^T&=z^T\ltimes (I_t\otimes A), \\
    z\ltimes A&=(I_t\otimes A)\ltimes z. 
  \end{align*}
\end{lemma}

Next, we reveal the deterministic essence of the three definitions of observability of PBNs studied in
\cite{Zhao2015ObservabilityPBN,Zhou2019ObservabilityPBN,Fornasini2020ObserReconPBN,Yu2021ObservabilityBooleanNetwork}.

\begin{remark}\label{rem6:synthesisConObser}  
	The PBNs studied in \cite{Zhao2015ObservabilityPBN,Zhou2019ObservabilityPBN,Fornasini2020ObserReconPBN,Yu2021ObservabilityBooleanNetwork} are as follows:
	\begin{equation}\label{PBN:synthesisConObser}
		\begin{split}
			x(t+1) &= L_{\sigma(t)} x(t),\\
			y(t) &= Hx(t),
		\end{split}
	\end{equation}
	where $\sigma:\N\to\llb 1,s \rrb$ with $s\in\Z_+$ is an independent and identically distributed process
	and at each time $t\in\N$, $\Prob\{\sigma(t)=i\}=p_i$, $i\in\llb 1,s \rrb$, with $p_i>0$\footnote{In
	\cite{Zhao2015ObservabilityPBN,Zhou2019ObservabilityPBN,Fornasini2020ObserReconPBN,Yu2021ObservabilityBooleanNetwork},
	the authors considered $p_i\ge0$. Here we consider $p_i>0$ with loss of generality because all 
	structure matrices $L_{i}$ with $p_i=0$ could be removed equivalently.} and $\sum_{i=1}^{s}
	p_i=1$, $[p_1,\dots,p_s]=:{\bf p}$ is the probability distribution of $\sigma$; 
	$L_1,\dots,L_s\in\LM_{2^n\times 2^n}$ are the system matrices; $H\in\LM_{2^q\times 2^n}$ is 
	the output matrix; $x(t)\in\Dt_{2^n}$; $y(t)\in\Dt_{2^q}$.

	Because $\sigma$ is independent and identically distributed, \eqref{PBN:synthesisConObser} can be reformulated
	as the following BCN:
	\begin{equation}\label{PBN_deterministic:synthesisConObser}
		\begin{split}
			x(t+1) &= [L_1,\dots,L_s] u(t) x(t),\\
			y(t) &= Hx(t),
		\end{split}
	\end{equation}
	where $u(t)\in\Dt_s$, the other variables are the same as above.

	Let ${\bf x}(\theta;\sigma,x_0)$ (resp., ${\bf y}(\theta;\sigma,x_0)$) be any of the admissible state
	(resp., output) sequences of \eqref{PBN:synthesisConObser}
	starting from initial state $x_0\in\Dt_{2^n}$ on discrete time set $\llb 0,\theta \rrb$.

	A PBN~\eqref{PBN:synthesisConObser}  is called observable in probability
	on $\llb 0,\theta \rrb$ with $\theta\in\N$ \cite{Zhao2015ObservabilityPBN} if for every two different $x_0,x_0'\in\Dt_{2^n}$,
	it holds that\footnote{Note
	that in ${\bf y}(\theta;\sigma,x_0)$ and ${\bf y}(\theta;\sigma,x_0')$, the only difference lies in 
	the different initial states. At each time step, in both of them, the switching signal $\sigma$
	takes the same value. Hence $\Prob\{{\bf y}(\theta;\sigma,x_0) \ne {\bf y}(\theta;\sigma,x_0)\}=0$
	for any $x_0\in\Dt_{2^n}$. Note also that $\Prob\{{\bf y}(\theta;\sigma,x_0) \ne {\bf y}(\theta;\sigma,x_0')\}$
	increases as $\theta$ increases and its limit need not be reached in finite time steps.}
	\begin{align}\label{eqn40:synthesisConObser}
		\Prob\{{\bf y}(\theta;\sigma,x_0) \ne {\bf y}(\theta;\sigma,x_0')\}>0.
	\end{align}

	By definition, \eqref{eqn40:synthesisConObser} holds, if and only if, in the corresponding BCN
	\eqref{PBN_deterministic:synthesisConObser}, $x_0$ and $x_0'$ have a distinguishing input sequence of 
	length $\theta$. 
	Hence a PBN
	\eqref{PBN:synthesisConObser} is observable in probability on $\llb 0,\theta \rrb$ with
	$\theta\in\N$ if and only if the BCN \eqref{PBN_deterministic:synthesisConObser} satisfies
	Definition~\ref{def4_observability} and additionally all pairs of different initial states $x_0,x_0'$
	have a distinguishing input sequence of fixed length $\theta$.

	A PBN~\eqref{PBN:synthesisConObser} is called finite-time observable in	probability \cite{Zhou2019ObservabilityPBN}
	if there is $\theta\in\N$ such that \eqref{PBN:synthesisConObser} is observable in probability on $\llb 0,\theta
	\rrb$. Then a PBN \eqref{PBN:synthesisConObser} is finite-time observable in probability if and only
	if the BCN \eqref{PBN_deterministic:synthesisConObser} satisfies Definition~\ref{def4_observability}.

	A PBN~\eqref{PBN:synthesisConObser} is called asymptotically observable in distribution 
	\cite{Zhou2019ObservabilityPBN} if for  every two different $x_0,x_0'\in\Dt_{2^n}$,
	it holds that
	\begin{align}
		\lim_{\theta\to\infty}\Prob\{{\bf y}(\theta;\sigma,x_0) \ne {\bf y}(\theta;\sigma,x_0')\}=1.
	\end{align}

	If $x_0$ and $x_0'$ have no distinguishing input sequence in \eqref{PBN_deterministic:synthesisConObser},
	then $\Prob\{{\bf y}(\theta;\sigma,x_0) \ne {\bf y}(\theta;\sigma,x_0')\}=0$ for any
	$\theta\in\N$. Hence in order to make \eqref{PBN:synthesisConObser} asymptotically observable in distribution,
	every pair of different states must have at least one distinguishing input sequence in 
	\eqref{PBN_deterministic:synthesisConObser}, i.e., \eqref{PBN_deterministic:synthesisConObser} satisfies 
	Definition~\ref{def4_observability}. One can see that \eqref{PBN_deterministic:synthesisConObser} satisfies 
	Definition~\ref{def4_observability} if and only if\footnote{This necessary and sufficient 
	condition is exactly the result shown in {\cite[Theorem~3.7]{Zhang2016ObservabilityofBCN},
	\cite[Theorem~5.7]{Zhang2014ObservabilityofBCNCCC}},
	\cite[Theorem~3.5]{Cheng2016NoteonObservabilityBCN}, and \cite[Algorithm~1]{Zhu2018ObservabilityBCN}.} in the observability graph of
	\eqref{PBN_deterministic:synthesisConObser}, for every non-diagonal vertex $v$, either $\outdeg(v)<2^n$\footnote{In this case, any $u$ in
	$\Dt_s\setminus\bigcup_{(v,\bar v)\in\E}\W((v,\bar v))$ is a distinguishing input sequence
	of $x_0$ and $x_0'$, where $\{x_0,x_0'\}=v$.}
	or some vertex $v'$ with $\outdeg(v')<2^n$ is reachable from $v$\footnote{
	In this case, any input sequence $Uu$ is a distinguishing input sequence of $x_0$ and $x_0'$, where
	$\{x_0,x_0'\}=v$, $U$ is the input sequence of any path from $v$ to $v'$, $u\in
	\Dt_s\setminus\bigcup_{(v',\bar v)\in\E}\W((v',\bar v))$.}.
	In order to make \eqref{PBN:synthesisConObser} asymptotically observable in distribution, one additionally
	must have in the observability graph of \eqref{PBN_deterministic:synthesisConObser}, there is no path from any 
	non-diagonal vertex to any diagonal vertex, because for a non-diagonal vertex $v=\{x_0,x_0'\}$ from which 
	there is a path to some diagonal vertex, there exists $\beta<1$ such that for any $\theta$
	greater than the length of the path, $\Prob\{{\bf y}(\theta;\sigma,x_0) \ne {\bf y}(\theta;\sigma,x_0')\}
	\le\beta$. Then $\lim_{\theta\to\infty}\Prob\{{\bf y}(\theta;\sigma,x_0) \ne {\bf y}(\theta;\sigma,x_0')\}
	\le\beta<1$.
	
	Conversely, assume that (\romannumeral1) \eqref{PBN_deterministic:synthesisConObser} satisfies 
	Definition~\ref{def4_observability} and (\romannumeral2) in the observability graph $\Gr_o:=(\V,\E,\W)$ of
	\eqref{PBN_deterministic:synthesisConObser},
	there is no path from any non-diagonal vertex to any diagonal vertex. We endow the edges of $\Gr_o$ with
	probabilities according to the probability distribution ${\bf p}=[p_1,\dots,p_s]$ as in 
	\eqref{PBN:synthesisConObser}: for all $v,v'\in\V$ with $(v',v)\notin\E$, $p_{v,v'}=0$;
	for all $v,v'\in\V$ with $(v',v)\in\E$, $p_{v,v'}=\sum_{i\in\llb 1,s \rrb, \dt_s^i\in\W((v',v))} p_i$.
	Denote the set of diagonal vertices and the set of non-diagonal vertices of $\Gr_o$ by $\V_d$
	and $\V_{nd}$, respectively. Then by (\romannumeral2), for all $v\in\V_{d}$ and $v'\in\V_{nd}$, $p_{v,v'}=0$.
	Denote the adjacency matrix of the non-diagonal subgraph of $\Gr_o$ by $M_{\V_{nd}}=(p_{v,v'})_{v,v'\in\V_{nd}}$.
	Then also by (\romannumeral2), the sum of the $v$-th column of $(M_{\V_{nd}})^{\theta}$ is equal to $\Prob\{
	{\bf y}(\theta;\sigma,x_0) = {\bf y}(\theta;\sigma,x_0')\}$, where $\{x_0,x_0'\}=v$. By (\romannumeral1)
	(i.e., for every $\{x_0,x_0'\}$ in $\V_{nd}$, $x_0$ and $x_0'$ have a distinguishing input sequence in
	$\Gr_o$), there exists $\bar\theta\in\Z_+$ such that in $(M_{\V_{nd}})^{\bar\theta}$, the sum of each column 
	is less than $1$. Then the spectral radius of $(M_{\V_{nd}})^{\bar\theta}$ is less than $1$, so is $M_{\V_{nd}}$. 
	Hence $\lim_{\theta\to\infty}(M_{\V_{nd}})^{\theta}$ has all entries equal to $0$. That is, for all $x_0,x_0'$
	with $(x_0,x_0')\in\V_{nd}$, $\lim_{\theta\to\infty}\Prob\{{\bf y}(\theta;\sigma,x_0) = {\bf y}(\theta;\sigma,
	x_0')\}=0$, $\lim_{\theta\to\infty}\Prob\{{\bf y}(\theta;\sigma,x_0) \ne {\bf y}(\theta;\sigma,
	x_0')\}=1$. For all $x_0,x_0'$ with $Hx_0\ne Hx_0'$, $\Prob\{{\bf y}(\theta;\sigma,x_0) \ne {\bf y}(\theta;\sigma,
	x_0')\}=1$ for any $\theta\in\N$. Then \eqref{PBN:synthesisConObser} 
	is asymptotically observable in distribution. 

	Based on the above discussion, a PBN \eqref{PBN:synthesisConObser} is asymptotically observable in distribution
	if and only if the corresponding BCN \eqref{PBN_deterministic:synthesisConObser} satisfies 
	Definition~\ref{def4_observability} and in the observability graph of \eqref{PBN_deterministic:synthesisConObser},
	there is no path from any non-diagonal vertex to any diagonal vertex\footnote{This is exactly the result in 
	\cite[Proposition~1]{Yu2021ObservabilityBooleanNetwork}.}.

	The above three conclusions show that the three definitions of observability in probability
	on $\llb 0,\theta \rrb$ with $\theta\in\N$, finite-time observability in probability, and asymptotic 
	observability in distribution are rather close to each other and do not depend on probability distributions 
	of the stochastic switching
	signal $\sigma$. Formally, given two probability distributions ${\bf p}^i$, $i=1,2$, for 
	PBN~\eqref{PBN:synthesisConObser}, one has PBN~\eqref{PBN:synthesisConObser} with ${\bf p}={\bf p}^1$ is
	observable if and only if PBN~\eqref{PBN:synthesisConObser} with ${\bf p}={\bf p}^2$ is observable,
	both in the sense of any one of the above three definitions.

Moreover, \cite[Lemma~3]{Yu2021ObservabilityBooleanNetwork} is exactly 
\cite[Theorem~6]{Moore1956}, and the central idea therein is the equivalence relation $\sim_k$ for a BCN 
for which two states have 
the relation if and only if they do not have any distinguishing input sequence of length $k$, also see
\cite[Remark~4.1]{Zhang2020bookDDS}.
This idea was also used in \cite{Fornasini2013ObservabilityReconstructibilityofBCN} (see Eqn. (4) therein)
to study 
observability and reconstructibility of BCNs, as well as in
\cite{Li2021EquivalentConditionDDP_BCN,Li2021ObservabilityDecompositionBCN}
to give a necessary and sufficient condition for the disturbance decoupling problem and to solve the observability
(Definition~\ref{def4_observability}) decomposition problem, in BCNs. In addition, there exist two mistakes in 
\cite[Table~\rom{2}]{Yu2021ObservabilityBooleanNetwork}, the time complexity of Algorithm~3.4 of [16]
in the table (i.e., of \cite{Cheng2016NoteonObservabilityBCN} in the current paper)
is not $O(2^{6n})$ but $O(2^{2n+m})$, the time complexity of Algorithm~1 of [18] in
the table (i.e., of \cite{Zhu2018ObservabilityBCN} in the current paper) is
not $O(2^{4n+m})$ but also $O(2^{2n+m})$. That is, the algorithms obtained in 
\cite{Cheng2016NoteonObservabilityBCN,Zhu2018ObservabilityBCN} are more efficient than 
\cite[Algorithm~1]{Yu2021ObservabilityBooleanNetwork}.

In \cite{Yu2021ObservabilityBooleanNetwork}, observability of switched Boolean networks was 
mentioned. As one can easily see, observability of switched Boolean control networks can be equivalently
transformed to observability of BCNs if inputs and switching signals of switched Boolean control networks
are with the same quantifier (either $\exists$ or $\forall$). For a given 
switched Boolean control network, one could regard the Cartesian product
of the set of inputs and the codomain of the switching signal as the new set of inputs so that a new
equivalent Boolean control network is obtained. The results in 
\cite{Zhang2016ObservabilitySBCN_NAHS} showed that Definition~\ref{def7_observability} extended to switched 
Boolean control networks for which the switching signals are with the $\exists$ quantifier is actually
Definition~\ref{def7_observability} of BCNs. As a sequence, controllability of switched Boolean control networks
studied in \cite{Li2012ControllabilitySwitchBCN} (in which inputs and switching signals are both with
$\exists$ quantifier) is actually controllability of BCNs.
\end{remark}

\end{document}